\documentclass{article}
\usepackage{graphicx,amsmath,amsfonts,amsthm,enumitem,comment,amssymb} 
\usepackage{amssymb}
\usepackage{mathtools}
\usepackage{geometry}[margin=1in]
\newtheorem{theorem}{Theorem}[section]
\newtheorem{lemma}{Lemma}[section]

\newtheorem{definition}{Definition}[section]
\newtheorem{proposition}{Proposition}[section]

\newcommand\enumar{\begin{enumerate}[noitemsep,label={\arabic*.}]}
\newcommand\enumrom{\begin{enumerate}[noitemsep,label={(\roman*)}]}
\newcommand\enumalph{\begin{enumerate}[noitemsep,label={(\alph*)}]}
\newcommand\enumend{\end{enumerate}}
\newcommand\itemgo{\begin{itemize}}
\newcommand\itemend{\end{itemize}}

\newcommand{\N}{\mathbb{N}}
\newcommand{\R}{\mathbb{R}}
\newcommand{\1}{\mathbf{1}}

\newcommand{\ep}{\epsilon}
\newcommand{\dlt}{\delta}
\newcommand{\ren}{\text{ren}}

\newcommand{\flsh}{\text{flsh}}
\newcommand{\pin}{\text{pin}}
\newcommand{\lf}{\lfloor}
\newcommand{\rf}{\rfloor}

\newcommand{\F}{\mathcal{F}}
\newcommand{\exit}{\text{ex}}
\newcommand{\lbd}{\lambda}
\newcommand{\nid}{\noindent}
\newcommand{\I}{\mathcal{I}}

\begin{document}

\begin{center}
\textbf{\large Time of appearance of a large gap in a dynamic Poisson point process}

\vspace{7mm}

 Eric Foxall\textsuperscript{1}, Clément Soubrier\textsuperscript{2}, 
\end{center}
\vspace{5mm}
{\small
$^{1}$ Department of Mathematics, University of British Columbia Okanagan, Kelowna, BC V1V 1V7, Canada\\
$^{2}$ Department of Mathematics, University of British Columbia, Vancouver, BC V6T 1Z4, Canada\\
}
\vspace{5mm}

\begin{abstract}
We study the distribution of the `gap time', the first time that a large gap appears, in the spatial birth and death point process on $[0,1]$ in which particles are added uniformly in space at rate $\lbd$ and are removed independently at rate $1$, as a function of the parameter $\lbd$ and the specified gap size function $w_\lbd$ as $\lbd\to\infty$. If $w_\lbd$ is a large enough multiple of the typical largest gap $(\log(\lbd)+O(1))/\lbd$ and the initial distribution has a high enough local density of particles and not too many particles in total, then the gap time, scaled by its expected value, converges in distribution to exponential with mean $1$. If in addition $\limsup_\lbd w_\lbd < 1$ then the expected time scales like $e^{\lbd w_\lbd}/(\lbd^2 w_\lbd(1-w_\lbd))$.\\

\end{abstract}

\noindent {\footnotesize MSC 2020: 60K35, 60G55, 60F10\medskip \\ 
\noindent Keywords:\\ spatial birth and death process, dynamic Poisson point process, regenerative process, rare event, hitting time}\\

\normalsize

\section{Introduction}\label{section:Introduction}

Consider the following family $X_\lambda = (X_\lambda(t))_{t \in \R_+}$ of spatial birth and death processes, indexed by $\lambda>0$ and taking values in finite subsets of $[0,1]$ (see Figure \ref{fig:model}): at exponential rate $\lambda$, a uniform random point is added to $X_\lambda$, and each particle in $X_\lambda$ is independently removed at rate $1$. That is,
\begin{equation}
\begin{cases} 
X_\lambda \to X_\lambda \cup \{u\} & \text{at rate} \ \lambda,\\
\text{for each} \ x \in X_\lambda, \ X_\lambda\to X_\lambda\setminus \{x\}  & \text{at rate} \ 1,\end{cases}
\label{eq:proc}
\end{equation}
where $u$ is  independently sampled from Uniform$[0,1]$ when an addition occurs. This process can also be seen as a \textit{dynamic Poisson point process}: 
if $X_\lambda(0)=\emptyset$, then for all $t>0$, $X_\lambda(t)$ is a Poisson point process with intensity $\lambda(1-e^{-t})$ (see Lemma \ref{lem:renprop}), and the unique stationary distribution of $X_\lambda$ is the probability measure $\pi_\lbd$ of a Poisson point process on $[0,1]$ with intensity $\lambda$ (see Lemma \ref{lem:dynerg}). We will often use the shorthand `dynamic Ppp$(\lbd)$' for the process $X_\lbd$ and `Ppp$(\lbd)$' for its stationary distribution, and `dynamic Ppp' when referring to the family of processes $(X_\lbd)_{\lbd>0}$.\\

In this article, we characterize the time of the first appearance of a large gap of a given minimum size $w_\lbd$, where a gap at time $t$ is defined to be any interval in $[0,1]$ that is disjoint from $X_\lambda(t)$. In other words, we study the hitting time of $A(w_\lambda)$, where  
\begin{align}\label{eq:Aw}  
A(w) = \{X\colon X\cap (y,y+w)=\emptyset \ \text{for some} \ y \in [0,1-w]\}.
\end{align}

Note that, in the definition of $A(w)$, the location of the gap is not specified; if it was, the problem would reduce to the hitting time of $0$ for the $\N$-valued birth and death process that gives the number of particles in the specified interval.\\

With respect to $\pi_\lbd$, the largest gap is approximately the maximum of $\lambda$ independent exponential random variables with mean $1/\lambda$, which scales like $\log(\lambda)/\lambda+O(1)$. Thus, for a gap of size $\ge w_\lbd$ to be large, $w_\lbd$ should be larger than $\log(\lbd)/\lbd$. For our estimates to work, a comfortable lower bound is that $\liminf_\lbd \lbd w_\lbd /\log(\lambda)\ge C$ for large enough $C>0$.\\

Our main findings are that, for suitable initial distributions,
\enumrom
\item the hitting time of $A(w_\lbd)$, normalized by its expected value, converges in distribution to an exponential random variable, and
\item the expected hitting time scales like a specific function of $\lbd$ and $w_\lbd$ that we compute.
\enumend

Our notion of `suitable initial distribution' is such that the process is unlikely to develop a large gap before its distribution approaches the stationary distribution. For this to hold, the initial distribution cannot develop a large gap in a short time, and cannot have too many particles, in the sense of conditions (c)-(d) in Theorem \ref{thm:main} below. The following definition specifies a condition that is sufficient for the process not to develop a large gap in a short time.

\begin{definition}\label{def:alphadense}
Let $X\subset [a,b]$ and $w>0$. Then $X$ is \emph{$\alpha$-dense at scale $w$} if $|X\cap (y,y+w)| \ge \alpha w$ for every $y \in [a,b-w]$. 

\end{definition}

The following is our main result.

\begin{theorem}\label{thm:main}
With $X_\lambda$ and $A(w_\lbd)$ as above, let $\tau_{A(w)}=\inf\{t\colon X_\lambda(t) \in A(w)\}$ denote the $w_\lambda$-gap time. For each $\alpha>0$ there is $C>0$ so that if
\enumalph
\item $\liminf_\lambda w_\lambda /\log(\lambda) \ge C$,
\item $\limsup_\lambda w_\lambda<1$,
\item $P(X_\lambda(0) \ \text{is} \ \alpha \lbd\text{-dense at scale } w_\lambda/2)\to 1$ as $\lbd\to\infty$ and
\item $P(\log |X_\lbd(0)|\le C_1\lbd)\to 1$ for some $C_1>0$ as $\lambda\to\infty$
\enumend
then
\enumrom
\item for every $t>0$, $P(\tau_{A(w_\lbd)}/E[\tau_{A(w_\lbd)}] > t) \to e^{-t}$ as $\lbd\to\infty$ and
\item $E[\tau_{A(w_\lbd)}] \sim e^{\lambda w_\lbd}/(\lambda^2w_\lbd(1-w_\lbd))$.
\enumend
\end{theorem}

Statement (i) says that the asymptotic distribution is exponential and statement (ii) gives the scaling of the expected value, which is equal to $1/q_\exit(A(w_\lbd))$ where $q_\exit(A(w_\lbd))\sim \lbd w_\lbd\pi_\lbd(A(w_\lbd))$ (see Lemma \ref{lem:exitest}) is the stationary exit rate from the event $A(w_\lbd)$ and $\pi_\lbd(A(w_\lbd)) \sim \lbd (1-w_\lbd)e^{-\lbd w_\lbd}$ is the stationary probability of $A(w_\lbd)$. The factor $\lbd w_\lbd$ in $q_\exit(A(w_\lbd))$ corresponds to the rate at which a point appears in a gap of size $w_\lbd$. Condition (a) ensures that $A(w_\lbd)$ is sufficiently rare, while (c)-(d) ensure respectively that $X_\lambda$ is unlikely to hit $A(w_\lbd)$ right away and doesn't take too long to reach equilibrium. Condition (b) ensures that statement (ii) has the given form; if, for example, $w_\lbd=1$ for all $\lbd$ (a case that is much simpler to handle) then $\pi(A(w_\lbd))=e^{-\lbd}$ and $1/q_\exit(A(w_\lbd)) \sim e^\lbd/\lbd$.\\

To visualize the dynamic Ppp and interpret the theorem, consider the following physical model (see Figure \ref{fig:model}): suppose  particles stochastically and independently attach (with rate $\lambda$) and detach (with rate 1) to a domain (e.g. molecules binding to the membrane of a cell). 
The binding positions can then be modelled in 1D by our stochastic process \eqref{eq:proc}, with the domain being [0,1]. Once attached, suppose that particles also bind with other attached particles that are within distance less than $w_\lbd$ to form a cluster. Assuming we start with one cluster covering the whole domain, the aforementioned hitting time then represents the first time when this cluster will split. As long as splitting requires a gap that is some large enough number of times as large as the typical largest gap (condition (b)), then in the large $\lambda$ regime ($\lambda$ being the average number of particles at stationarity), the theorem says that this time is exponential, and gives the rate, which quantifies the stability of the cluster.\\

\begin{figure}[!ht]\label{fig:model}
	\centering
        \includegraphics[width=0.99\textwidth]{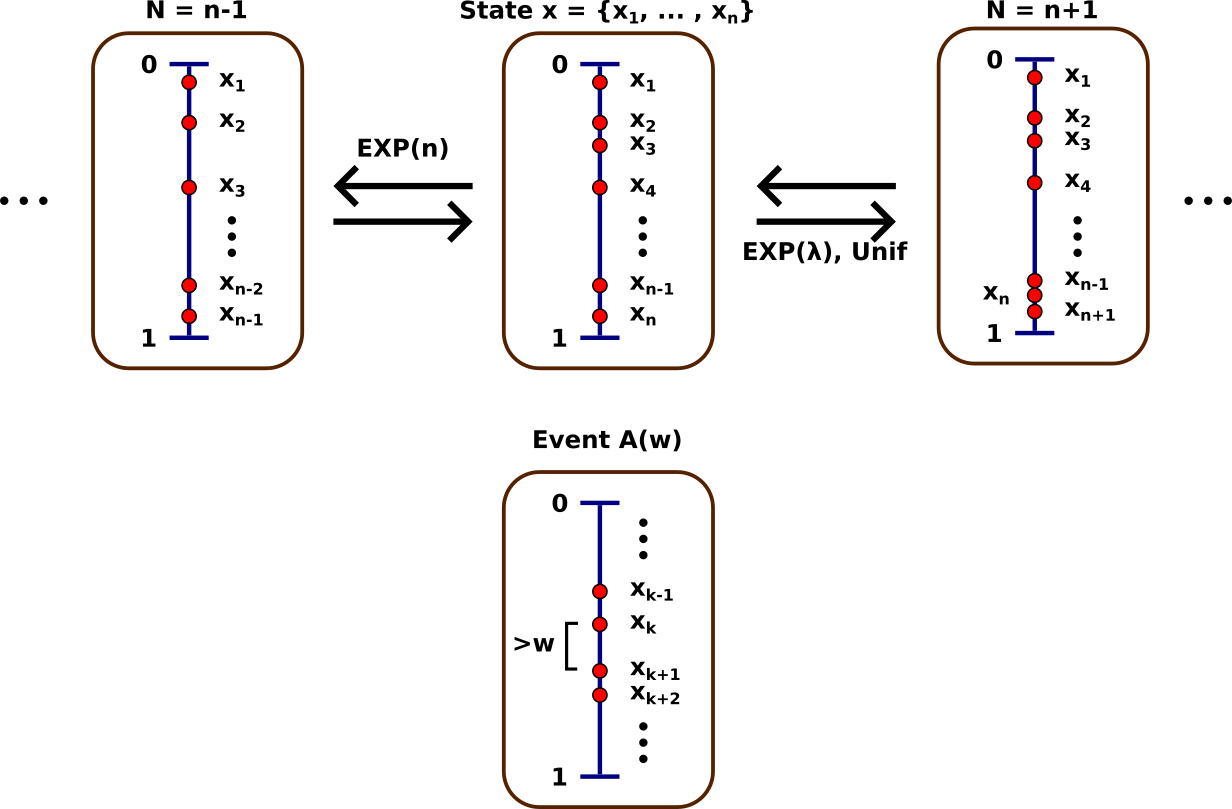}
	\caption{Schematic representation of the dynamic Poisson point process, with a large gap event $A(w)$. }
\end{figure}

In the dynamic Ppp, particles are added and removed but the particles do not move. Concerning the opposite situation, there are various examples of first passage time problems for rare events associated with particles undergoing Brownian motion, with applications to single bio-molecules reaching small traps in various domains \cite{ coombs2009diffusion,schuss2007narrow}, or signalling pathways activation requiring large numbers of molecular binding \cite{Daoduc2010thres}. Perhaps the most closely related example is the first time for a fixed region (representing a gap) to become emptied of a large fixed number $N$ of Brownian particles, a problem that was studied in \cite{newby2016first} and for which detailed asymptotics were obtained for the expected hitting time as $N\to\infty$.\\

The dynamic Ppp is a family of spatial birth and death processes (SBDs). An SBD can be formally defined, roughly following \cite{garcia2006spatial}, as any continuous-time Markov process whose state space is a collection of finite or countable subsets of a complete, separable metric space $S$, and whose transitions are characterized by local birth and death rates $b(x,X), d(x,X)$. These rates are defined with respect to a reference measure $\beta$ on $S$, such that the rate at which a new point appears in $B\subset S$, when the state of the process is $X$, is equal to $\int_{x \in B} b(x,X)\beta(dx)$, and the rate at which each point $x \in X$ is removed is equal to $d(x,X)$. The dynamic Ppp has $S=[0,1]$ and is perhaps the simplest family of SBDs, as $\beta$ is Lebesgue measure and $b=\lbd, \ d=1$ do not depend on $x,X$; in fact, for each $\lbd$, $X_\lbd$ can be viewed as a spatial version of the $M/G/\infty$ queue, a queuing system with infinitely many servers (spots where particles can attach in our physical representation) and no waiting time (so particles instantaneously attach as they arrive) \cite{benevs1957fluctuations,tijms2003first}. The study of SBDs was initiated by Preston in 1975 \cite{preston1975spatial} and, as described, e.g., in \cite{garcia2006spatial}, soon found use in the simulation of spatial point processes, since SBDs can be designed to have specific stationary distributions to which they converge quickly. SBDs have also been used to model the temporal evolution of spatial patterns, including earthquakes epicentres clustering over time \cite{illian2008statistical,ogata1998space}, forest fires pattern evolution, storm centre dynamics \cite[Chapter 15.4]{daley2008introduction} or wireless network behaviour \cite{sankararaman2017spatial}. Because of their applications in the simulation of point processes, theoretical work on SBDs has largely focused on the following two categories: (i) existence and uniqueness of the process and its stationary distribution, characterization as the solution to certain stochastic equations, ergodicity and rate of convergence to the stationary distribution \cite{bezborodov2017maximal,bezborodov2022spatial,garcia2006spatial} and (ii) results for statistical inference \cite{cressie2015statistics,fonseca2017dynamic,gu2021spatio,illian2008statistical,lieshout_theory_2019}. The results of the present article fall within the category of sample path properties of a family of SBDs in the high-density limit (taking $\lbd\to\infty$). There has been some other work in this context, specifically in proving functional central limit theorems for particle counts in families of sets on which the expected number of points has a positive, finite limit (see, e.g., \cite{onaran2023functional} and references therein). Although there has been some work studying large deviations for nonspatial (i.e., $\N$ or $\N^d$-valued) birth and death processes \cite{assaf2017wkb,chan1998large,lazarescu2019large}, and for spatial point processes (see e.g., the next paragraph), to our knowledge, the present work is the first study of a non-trivially spatial rare event for a family of SBDs in the high-density limit.\\

As previously mentioned, the dyanmic Ppp admits the probability measure of a Poisson point process as its unique stationary distribution. Point processes, also called random point fields, are random variables whose realizations are locally finite collections of points \cite{lieshout_theory_2019}.
They usually model temporal or spatial patterns \cite{cressie2015statistics}, such as failure occurrences in computer reliability \cite{asmussen2008asymptotic}, arrival times in a queue \cite{benevs1957fluctuations}, geological patterns \cite{illian2008statistical} or tree distribution in a forest \cite{chiu2013stochastic}.
In homogeneous Poisson point processes, the random variable $G$ of the largest gap between two consecutive points has an analytically known cumulative distribution function \cite[2.3]{fan2000asymptotic}. However, due to its complexity, this formula does not allow to find asymptotic estimations as $\lambda\to \infty$. More generally, applying the strong law of large numbers to $G$ yields \cite[Lemma 6 and Theorem 10]{deheuvels1985erdos} $G\sim \log(\lbd)/\lbd$ a.s.
Large deviations estimates yield bounds on the cumulative distribution function of $G$ \cite{omwonylee2020general}. Under the assumptions $ w_\lambda\gg \log(\lbd)/\lbd$   as $\lambda\to +\infty$ and $\limsup_{\lambda}w_\lambda<1$, these bounds read (Theorem 1.2 with $t=\lambda$, $l=\lambda w_\lambda$ and $L(\lambda) = \lambda G$):
\[\frac{\lambda(1-w_\lambda)}{(1-e^{-1})^2} e^{-\lambda w_\lambda}\lesssim P (G\geq w_\lambda) \lesssim \lambda(1-w_\lambda)e^{-\lambda w_\lambda}.\]
While the focus of our paper is not on studying these bounds, we prove in Lemma \ref{lem:piAw} below that under somewhat weaker assumptions on $w_\lbd$, the upper bound is asymptotically reached. \\

For rare events, i.e., sequences of subsets $(A_m)$ of a state space with stationary probability $\pi(A_m)$ tending to $0$, convergence of the hitting time $\tau(A_m)$, rescaled by its expected value, to an exponential random variable (asymptotic exponentiality) has been studied for ergodic dynamical systems, \cite{abadi2004sharp,pene2020spatio,zweimuller2019hitting} general classes of Markov processes, \cite{cogburn1985distribution, glasserman1995limits, glynn2011exponential, iscoe1994asymptotics,keilson1979rarity}, and specific processes, such as interacting particle systems \cite{asselah1997occurrence} including exclusion \cite{asselah1997sharp} and contact \cite{mountford2016exponential} processes. Since the exponential distribution is memoryless and places no mass at zero, the essential condition for asymptotic exponentiality, from a given sequence $(\sigma_m)$ of initial distributions, is that the system forgets its past state on a shorter time scale than the hitting time of the rare event. This is generally broken up into two conditions: that on the shorter time scale, (i) the distribution of the process converges to $\pi$, or to an appropriate reference measure, such as a quasi-stationary distribution, and (ii) the process does not hit the rare event.
The only apparent exception to (i) is the approach taken in \cite{zweimuller2022hitting}, in which exponentiality is deduced from a distributional relation between the hitting time of $A_m$ with initial distribution $\pi$ and the return time to $A_m$ with initial conditional distribution $\pi(\cdot \mid A_m)$, although ergodicity of the process is still assumed. Condition (ii) is, generally, more or less directly assumed. An additional condition of (iii) compactness/uniform integrability of $\tau(A_m)$ is sometimes needed, when it does not follow from the context, to ensure that $E_{\sigma_m}[\tau(A_m)]$ is the time scale on which the distribution of $\tau(A_m)$ converges. 
Condition (i) can be imposed in multiple ways: bounds on the mixing time \cite{benois2013hitting}, the spectral gap of the generator \cite{iscoe1994asymptotics} or the return time to a specified state \cite{fernandez2015asymptotically}, existence of a \emph{strong metastability time} at which the distribution is quasi-stationary conditioned on not having hit the rare event, or the existence of i.i.d.~or one-dependent regeneration cycles, as defined and shown to exist for Harris recurrent Markov processes in \cite{sigman1990one}, on which the hitting probability of $A_m$ tends to $0$ \cite{cogburn1985distribution,glasserman1995limits,glynn2011exponential}. In this article, we take the regeneration approach, which is well-suited to the dynamic Ppp as the latter has a natural `refresh' mechanism described in Section \ref{sec:renew}. When condition (i) is established through short-time convergence to a reference measure, the memoryless property is typically verified directly by showing the rescaled tail probability function $t\mapsto P(\tau(A_m)/E_{\sigma_m}[\tau(A_m)] > t)$ is asymptotically multiplicative; the martingale approach used in \cite{fernandez2015asymptotically} for this purpose is notable for its elegance. When regeneration cycles are used, the exponential distribution for the first hitting cycle can be obtained as a limit of geometric random variables, e.g., \cite{glynn2011exponential}, and with some additional effort this can be transferred from the first hitting cycle to the first hitting time; this is the approach we take in Proposition \ref{prop:exp}. Most of the literature that we could find on asymptotic exponentiality for general (non-specific) processes concerns a single dynamical system or Markov chain, rather than a sequence or family of processes. The distinction between the two settings is somewhat analogous to that for weak laws of large numbers, which can be proved either for a single sequence, versus a triangular array. Of the three studies \cite{benois2013hitting,fernandez2015asymptotically,fernandez2016conditioned} of asymptotic exponentiality that we could find that pertain to a sequence of processes, all three restricted attention to finite state Markov chains. To the best of our knowledge, the theory contained in Section \ref{sec:theory} of the present article is the first general treatment of asymptotic exponentiality of the hitting time of rare events for a sequence of Markov processes each with infinite state space.\\

 For the purpose of determining what constitutes a shorter time scale than the hitting time of rare events, or for its own sake, in the context of the previous paragraph it is often of interest to characterize the scale of $E_{\sigma_m}[\tau(A_m)]$. When the discrete-time process of visits to $A_m$ has a stationary distribution $\mu_{A_m}$, it is convenient, when possible, to relate $E_{\sigma_m}[\tau(A_m)]$ to the stationary expected return time $E_{\mu_{A_m}}[\tau(A_m)]$. These generally have the same scaling when conditions (i)-(ii) of the previous paragraph hold for the return time to $A_m$ when the initial distribution is $\mu_{A_m}$, i.e. that on a shorter time scale than $E_{\mu_{A_m}}[\tau(A_m)]$, (i) the distribution of the process converges to the reference measure and (ii) there are no early returns to $A_m$, terminology that we borrow from \cite{cogburn1985distribution}. In this article, these conditions are handled in Proposition \ref{prop:tauex}. The usage of $E_{\mu_{A_m}}[\tau(A_m)]$ is convenient because, under mild conditions, it is equal to the reciprocal of the stationary rate of transition from $A_m^c$ into $A_m$, as shown in Proposition \ref{prop:mustat} of this article. This connection is used implicitly in both \cite{fernandez2015asymptotically}, where it relates to the definition of the rates $r_N$, and \cite{zweimuller2022hitting}, in which $\pi(A_m)$ (denoted $\mu(A_l)$ in that reference), which is used implicitly to rescale the hitting and return times, is equal to the stationary frequency of transitions from $A_m^c$ into $A_m$, as it is for any discrete-time process with ergodic stationary distribution $\pi$. Note that, in this article, we stop the process not when it returns to $A_m$ but when it exits $A_m$, i.e., we focus on transitions from $A_m$ to $A_m^c$, since it is more convenient once we need to do calculations for the dynamic Ppp.\\

The article is organized as follows. Section \ref{sec:theory} contains the general theory needed to deduce the asymptotic exponential distribution, and scaling, of the hitting time of rare events. In Section \ref{sec:thry1} we give some basic definitions and results related to regeneration times. In Section \ref{sec:thry2}, the main result is Proposition \ref{prop:exp}, in which we give conditions for the hitting time of a sequence of events $(A_m)$, associated to a sequence of ergodic Markov chains $(X_m)$, equipped with regeneration times at which $X_m$ has distribution $\nu_m$, to be asymptotically exponentially distributed, when the initial distribution is $\nu_m$. In Section \ref{sec:thry3} we give the results needed to characterize $E_{\nu_m}[\tau(A_m)]$, the expected hitting time of $A_m$, from initial distribution $\nu_m$, in terms of the stationary exit rate from $A_m$: in Proposition \ref{prop:mustat} we establish existence of the stationary exit distribution and relate the expected duration of the exit cycle (the trajectory segment from one exit time to the next) to the stationary exit rate, and in Proposition \ref{prop:tauex} we give conditions under which the expected hitting time from initial distribution $\nu_m$ scales like the expected duration of an exit cycle. In Section \ref{sec:renew} we define a `refresh time' for the dynamic Ppp, which we show in Lemma \ref{lem:renprop} is a regeneration time in the sense of Definition \ref{def:renew}, and in Lemma \ref{lem:renest} we establish several estimates for the refresh time. Lemma \ref{lem:countest}, proved in Section \ref{sec:Nest}, provides the estimates for the $\N$-valued counting process associated to the dynamic Ppp that are needed to prove Lemma \ref{lem:renest}. In Section \ref{sec:proof} we give the main steps to the proof of Theorem \ref{thm:main}, which consist mainly in verifying that the sufficient conditions established in Section \ref{sec:theory} hold for the dynamic Ppp; the details are given at the beginning of that section. Section \ref{sec:piest} contains estimates for the stationary distribution of the model, the Poisson point process on $[0,1]$. Section \ref{sec:Nest} contains estimates for the sample paths of the above-mentioned $\N$-valued counting process.

\section{Rare event theory}\label{sec:theory}

In this section, $(S,d)$ is a complete separable metric space and $X=(X_t)_{t\ge 0}$ is a pure jump strong Markov process with transition rate kernel $q(x,dy)$, i.e., such that for $x \in S$ and any Borel set $B\subset S$, for small $h>0$,
\[P(X(t+h) \in B \mid X(t)=x) = h \,q(x,B) + O(h^2).\]
We'll refer to $X$ as a Markov chain on $S$. Use $P_\psi, E_\psi$ for probability and expectation when $X(0)$ has distribution $\psi$.

\subsection{Stopping sequence, Regeneration time}\label{sec:thry1}

We will frequently refer to the increasing sequence of stopping times obtained by iterative application of a stopping time, defined as follows.

\begin{definition}[Stopping sequence, trajectory segments]\label{def:renseq}
Let $\tau$ be a stopping time for a Markov chain $X=(X_t)_{t\ge 0}$, viewing $\tau$ as an $\R_+$-valued function of the trajectory $X(\cdot)$. For $t\in \R_+$ let $\tau(t)$ denote the value of the stopping time for the shifted trajectory $X(t+\cdot)$. The \emph{stopping sequence} $(\tau_n)$ corresponding to $\tau$ is given by $t_0=0$,
\[\tau_n = t_{n-1} + \tau(t_{n-1})\]
and for each $n$, the \emph{$n^{\text{th}}$ trajectory segment $Y_n$ with respect to $(\tau_n)$} is defined by
\[Y_n(t)= X(\tau_{n-1}+t), \ 0<t <\tau_n-\tau_{n-1}.\]
\end{definition}

The stopping sequence is used to prove the following ergodic-theoretic result that applies to ergodic Markov chains with a stopping time $\tau$ that preserves a given distribution $\psi$, in the sense described below. The result is used in the proof of Lemma \ref{lem:prop1cond}. The notation $\ll$ refers to absolute continuity of measures.

\begin{definition}[Time shift operator]\label{def:timeshift}
Let $x(t)$ be a function defined for $t\in \R_+$ and let $t_0>0$. The $t_0$ time shift operator $F_{t_0}$ is defined by
\[(F_{t_0} x)(t) = x(t_0+t).\]
\end{definition}

\begin{lemma}[Excursion principle]\label{lem:excprinc}
Let $(X(t))$ be a Markov chain on $S$ with stationary distribution $\pi$ such that for some $t_0>0$, $P_\pi$ is ergodic with respect to the $t_0$ time shift $X\mapsto F_{t_0}X$. In addition, let $\tau,\psi$ be a stopping time and distribution such that if $X(0)$ has distribution $\psi$ then $\tau>0$ almost surely and $X(\tau)$ has distribution $\psi$. 
Let $T_A(t):=\int_0^t \1(X(s) \in A)ds$ denote the occupation time of $A\subset S$ up to time $t$. If $\psi\ll \pi$ and $E_\psi[\tau]<\infty$ then
\[\pi(A) = \frac{E_\psi[T_A(\tau)]}{E_\psi[\tau]}.\]
\end{lemma}

\begin{proof}
 As in Definition \ref{def:renseq} let $(\tau_n)$ be the stopping sequence corresponding to $\tau$ and $(Y_n)$ the corresponding trajectory segments, so that if $X(0)$ has distribution $\psi$ then by the strong Markov property applied to $X$, the sequence $(Y_n)_{n\ge 0}$ is stationary with respect to the shift map defined by $\Sigma((Y_0,Y_1,\dots,)) = (Y_1,Y_2,\dots)$. Let $w_n=\tau_n-\tau_{n-1}$ so that $w_n$ is determined by $Y_n$ for each $n$. The ergodic theorem implies that $P_\psi$-almost surely, as $n\to\infty$
\[n^{-1}\tau_n \to E_\psi[\tau \mid \I], \ \ n^{-1} T_A(\tau_n) \to E_\psi[T_A(\tau) \mid \I],\]
where $\I$ is the invariant $\sigma$-algebra for the shift map. If $E_\psi[\tau]<\infty$ then $E_\psi[\tau \mid \I] \in L^1(\psi)$ so $P_\psi(E_\psi[\tau \mid \I]<\infty)=1$, and since $P_\psi(\tau>0)=1$ by assumption, $P_\psi(E_\psi[\tau \mid \I]>0)=1$. Combining the above limits, $P_\psi$-almost surely, as $n\to\infty$
\begin{align}\label{eq:Aocclim}
\tau_n^{-1} T_A(\tau_n) = (n^{-1}\tau_n)^{-1} n^{-1}T_A(\tau_n) \to \frac{E_\psi[T_A(\tau) \mid \I]}{E_\psi[\tau \mid \I]}
\end{align}
and 
\begin{align}\label{eq:tnlim1}
\frac{t_{n+1}}{t_{n}} = \frac{t_{n+1}}{n+1}\frac{n+1}{n}\frac{n}{t_{n}} \to \frac{E_\psi[\tau \mid \I]}{E_\psi[\tau \mid \I]}=1.
\end{align}
We can strengthen \eqref{eq:Aocclim} to convergence of $t^{-1}T_A(t)$ as $t\to\infty$ as follows: with $n(t):= \max\{n\colon \tau_n\le t\}$, since $t\mapsto T_A(t)$ is non-decreasing
\[\frac{\tau_{n(t)}}{t}\tau_{n(t)}^{-1}T_A(\tau_{n(t)}) \le t^{-1}T_A(t) < \frac{\tau_{n(t)+1}}{t}\tau_{n(t)+1}^{-1}T_A(\tau_{n(t)+1}).\]
Since $\tau_n/n\to E_\psi[\tau]>0$, $n(t)\to\infty$ as $t\to\infty$. Since in addition, $ \tau_{n(t)}/\tau_{n(t)+1}\le \tau_{n(t)}/t \le 1$ and $1 \le \tau_{n(t)+1}/t \le \tau_{n(t)+1}/{\tau_{n(t)}}$ it follows that $\tau_{n(t)}/t \to 1$ and $\tau_{n(t)+1}/t \to 1$ as $n\to\infty$. Using this and \eqref{eq:Aocclim} in the above display, \eqref{eq:Aocclim} holds as $t\to\infty$ with $t$ in place of $\tau_n$.\\

Taking $X(0)$ to have distribution $\pi$, ergodicity of $\pi$ with respect to $F_{t_0}$ implies that
\[n^{-1}T_A(nt_0) \to E_\pi[T_A(t_0)] = \int_0^{t_0} E_\pi[\1(X(t) \in A)]dt = \int_0^{t_0} \pi(A) = t_0\pi(A),\]
so using this and the stronger form of \eqref{eq:Aocclim}, if $\psi \ll \pi$ then $P_\psi$-almost surely
\[\pi(A) = \lim_{n\to\infty}(t_0n)^{-1}T_A(nt_0) = \frac{E_\psi[T_A(\tau) \mid \I]}{E_\psi[\tau \mid \I]}.\]
Thus $E_\psi[T_A(\tau) \mid \I]= \pi(A) E_\psi[\tau \mid \I]$, and integrating
\[E_\psi[T_A(\tau)] = E_\psi[E_\psi[T_A(\tau) \mid \I]] = E_\psi[\pi(A) E_\psi[\tau \mid \I]] = \pi(A) E_\psi[\tau].\]
Solving for $\pi(A)$ gives the result.
\end{proof}

\begin{definition}[Regeneration time]\label{def:renew}
Let $(X(t))$ be a Markov chain on state space $S$. A \emph{regeneration time} and \emph{regeneration measure} are respectively a stopping time $\eta$ and a probability measure $\nu$ on $S$ such that, for any distribution of $X(0)$, $\eta>0$ almost surely, $X(\eta)$ has distribution $\nu$ and $X(\eta)$ is independent of $X(0)$. Letting $(\tau_n)$ denote the stopping sequence corresponding to $\eta$, the \emph{$n^\text{th}$ regeneration cycle} is the $n+1^\text{st}$ trajectory segment with respect to $(\tau_n)$ in the sense of Definition \ref{def:renseq}.
\end{definition}

Note that if $X=(X(t))_{t\ge 0}$ has a regeneration time and regeneration measure then it is one-dependent regenerative in the sense described in \cite{sigman1990one}. The following result is immediate from the definition and the strong Markov property and details the one-dependence property that we will need in what follows.

\begin{lemma}[One-dependence]\label{lem:onedep}
Let $\eta$ be a regeneration time with stopping sequence $(\tau_n)$. If intervals $\{n_i,\dots,m_i\}$ with $n_i<m_i$, $i=1,\dots,k$ are \emph{separated}, i.e., $m_i<n_{i+1}$ for $i=1,\dots,k-1$, then $Z_1,\dots,Z_k$ defined by $Z_i(t)=X(\tau_{n_i}+t)$, $0\le t < \tau_{m_i}-\tau_{n_i}$ are independent. Moreover, if $n_1\ge 1$ and for some $m$, $m_i=m$ for all $i$, then $Z_1,\dots,Z_k$ are identically distributed.
\end{lemma}

\subsection{Exponential limit distribution}\label{sec:thry2}

The following result gives sufficient conditions for a sequence of events corresponding to a sequence of Markov chains, each possessing a regeneration time, to have an asymptotically exponential distribution. It is the main result of this subsection.

\begin{proposition}\label{prop:exp}
Let $(X_m)_{m\in \N}$ be a sequence of Markov chains, each $X_m=(X_m(t))_{t\ge 0}$ taking values in $S_m$, and for each $m$ let $\eta_m$, $\nu_m$ be a regeneration time and regeneration measure for $X_m$, with $(t_{m,n})_n$ denoting the corresponding stopping sequence.\\

Let $A_m\subset S_m$, let $p_m = P_{\nu_m}(X_m(t) \in A_m \ \text{for some} \ t < \eta_m)$ and let $N_m=\inf\{n\colon \tau(A_m) < t_{m,n}\}$. For $\delta>0$ let $n_m(\delta)=\inf\{n\colon P_{\nu_m}(N_m \le n )\ge \delta\}$. Let $X_m(0)$ have distribution $\nu_m$ for each $m$ and suppose that
\enumrom
\item $p_m\to 0$ as $m\to\infty$, and
\enumend
if $\dlt_m\to 0$ slowly enough as $m\to\infty$ then 
\enumrom
\item[(ii)] $(t_{m,n_m(\delta_m)} / E_{\nu_m}[t_{m,n_m(\delta_m)}])_m$ is uniformly integrable, and
\item[(iii)] $E_{\nu_m}[t_{m,n_m(\delta_m)} \mid N_m \le n_m(\delta_m)] = O(E_{\nu_m}[t_{m,n_m(\delta_m)}])$.
\enumend
Then
\[P_{\nu_m}(\tau(A_m)/E_{\nu_m}[\tau(A_m)]>t) \to e^{-t} \ \ \text{as} \ \ m\to\infty.\]

\end{proposition}

\begin{figure}[!ht]\label{fig:prop_1}
	\centering
        \includegraphics[width=0.85\textwidth]{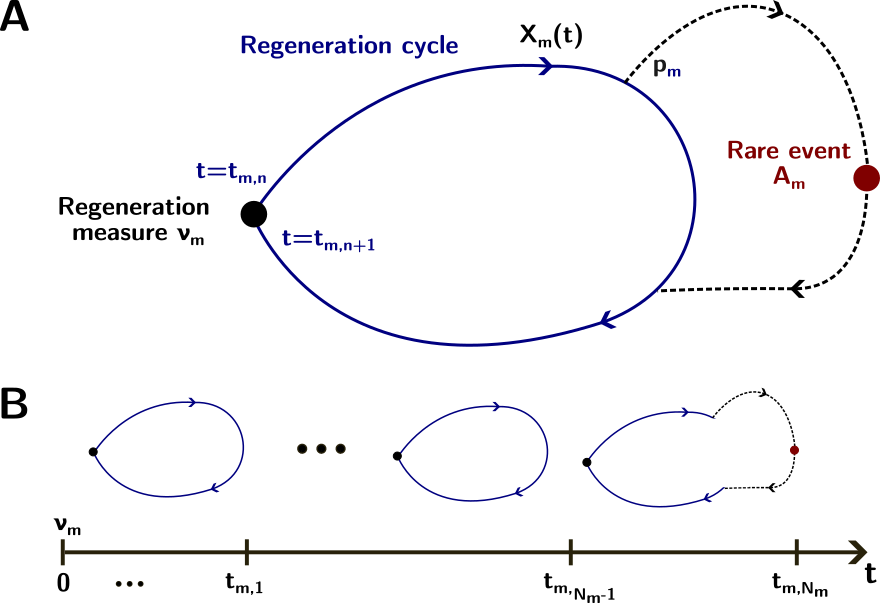}
    
	\caption{Schematic representation of the setup of Proposition \ref{prop:exp}. (A) We suppose that we can define regeneration cycles which are one-dependent. During each cycle, the process can hit $A_m$ with probability $p_m\to0$ \textit{(i)}. (B) $N_m$ counts the number of cycles before the process hits $A_m$.}
\end{figure}

We will need two more results in order to prove Proposition \ref{prop:exp}: an estimate that is used in another article proving a weak law of large numbers, and a simple result concerning uniform integrability.

\begin{lemma}[Estimate from the proof of the Theorem, Section 3, \cite{gutmart}, $p=1$ case.]\label{lem:gut} Let $(y_i)_{i=1}^k$ be a collection of $L^1$ random variables with natural filtration $(\F_i)$ and let $\mu_i=E[y_i \mid \F_{i-1}]$. Then there is $C>0$ so that for all $a>0$,
\[E\left[|\sum_{i=1}^k (y_i-\mu_i)|\right] \le C k^{1/2}a + C E\left[\sum_{i=1}^k |y_i-\mu_i|\1(|y_i-\mu_i|\ge a)\right]. \]

\end{lemma}

\begin{lemma}\label{lem:UIcond}
    Let $(U_m)$ be uniformly integrable and let $(V_m)$ have the distribution of $U_m$ conditioned on events $A_m$ satisfying $\liminf_m P(A_m)>0$. Then for some $m_0$, $(V_m)_{m\ge m_0}$ are uniformly integrable.
\end{lemma}

\begin{proof}
Let $\ep,m_0$ be such that $P(A_m)\ge \epsilon>0$ for all $m\ge m_0$. For each $a$ and every $m\ge m_0$,
\[E[\, |V_m|>a \,\1(|V_m|>a)\, ]\le \dfrac{ E[\, |U_m|>a \, \1(|U_m|>a)]}{\ep}.\]
By assumption, $\sup_{m\ge m_0} E[\, |U_m|>a\,  \1(|U_m|>a)]\to 0$ as $a\to\infty$, and from the above inequality the same is true of $\sup_{m\ge m_0}E[\, |V_m|>a \1(|V_m|>a)]$.
\end{proof}

We can now prove Proposition \ref{prop:exp}.

\begin{proof}[Proof of Proposition \ref{prop:exp}]
We begin with an outline. The idea is to break up time into intervals of sufficiently many regeneration cycles that the probability of hitting $A_m$ on each interval is approximately $\delta>0$, then to check on each interval whether the process hits $A_m$. Between intervals we leave a gap of one regeneration cycle on which we do not check for hitting $A_m$. By the one-dependence property described in Lemma \ref{lem:onedep}, the trajectory segments on successive intervals are independent, while the probability of hitting $A_m$ in any particular gap (and thus of not noticing that we hit it) is $p_m$ which $\to 0$ by (i). The index of the first interval on which $A_m$ is hit is geometrically distributed, while the durations of preceding intervals are conditionally i.i.d. Since $p_m =o(1)$, intervals are made up of large numbers of regeneration cycles, so the duration of the gaps, which each comprise a single regeneration cycle, is small by comparison. Letting $\dlt$ tend to $0$ slowly as $m\to\infty$,
\itemgo[noitemsep]
\item the geometric distribution, scaled by $\dlt$, converges to exponential,
\item a law of large numbers can, with the help of condition (ii), be applied to interval durations preceding the hitting interval, and
\item since the index of the first hitting interval diverges as $\dlt \to 0$, the duration of this interval is, by (iii), negligible compared to the sum of the preceding intervals' durations.
\itemend
This will show that the hitting time, scaled in some way, converges in distribution to exponential. With $L^1$ convergence of the relevant quantities, we can show the scaling is by the expected value, as desired, completing the proof.\\

Let $\alpha_{m,n}=\1(X_m(t) \in A_m \ \text{for some} \ t\in [t_{m,n-1},t_{m,n}))$, then by definition of regeneration time and the strong Markov property $P_{\nu_m}(\alpha_{m,n}=1) = p_m$ for all $n,m$. A union bound implies $P_{\nu_m}(N_m \le n+1) \le P_{\nu_m}(N_m\le n)+P_{\nu_m}(\alpha_{m,n+1}=1)$, so in particular $q_m(\dlt):=P_{\nu_m}(N_m \le n_m(\dlt))$ has $\dlt \le q_m(\dlt) \le \dlt+p_m$. Let $n_m'(\dlt)=n_m(\dlt)+1$, let $I_{m,i}(\dlt)=\{(i-1)n_m'(\dlt)+1,\dots,i n_m'(\dlt)-1)\}$ and let $a_{m,i}(\delta)=\1( \alpha_{m,n}=1 \ \text{for some} \ n \in \{(i-1) n_m'(\delta)+1,\dots, in_m'(\delta)-1\})$, then by definition of regeneration time and the strong Markov property, $P_{\nu_m}(a_{m,i}(\dlt)=1)=q_m(\dlt)$ for each $i$. Define $Y_{m,n}$ on $[0,t_{m,n}-t_{m,n-1})$ by $Y_{m,n}(t)=X_m(t_{m,n-1}+t)$. Since $a_{m,i}(\dlt)$ is determined by $(Y_{m,n}\colon n\in I_{m,i}(\dlt))$, by Lemma \ref{lem:onedep} and the definition of $n_m(\dlt)$ and $p_m$, $(a_{m,i}(\dlt))_i$ are independent, so 
\[k_m(\delta):=\inf\{i\colon a_{m,i}(\delta)=1\}\]
has distribution geometric$(q_m(\dlt))$. If $\delta_m\to 0$ slowly as $m\to\infty$ then $\delta_m k_m(\delta_m)$ converges in distribution to exponential with mean $1$. For each $\delta>0$, $N_m \le k_m(\delta)n_m'(\delta)-1$ almost surely, moreover, for fixed $k$ \[P_{\nu_m}(N_m\le (k_m(\dlt)\wedge k -1)n_m'(\dlt)) \le P_{\nu_m}(N_m = jn_m'(\dlt) \ \text{for some} \ j<k) \le (k-1)p_m \to 0 \text{ as } m\to\infty\] while $P_{\nu_m}(k_m(\dlt)>k)\to 0$ uniformly in $m$ as $k\to\infty$. Taking $k\to \infty$ slowly as $m\to\infty$, $P_{\nu_m}(N_m > (k_m(\delta)-1)n_m'(\delta) ) \to 1$ as $m\to\infty$ and correspondingly
\begin{align}\label{eq:tAmsqz}
&\tau(A_m) < t_{m,k_m(\dlt)n_m'(\dlt)-1} \ \ \text{almost surely and} \nonumber \\
&P(\tau(A_m)\ge t_{m,(k_m(\dlt)-1)n_m'(\dlt)})\to 1 \ \ \text{as} \ \ m\to\infty. 
\end{align}

Define
\begin{align*}
w_{m,i}(\dlt) &= t_{m,in_m'(\delta)-1} - t_{m,(i-1)n_m'(\delta)}, \ \ v_{m,i}(\dlt) = t_{m,in_m'(\delta)} - t_{m,(i-1)n_m'(\delta)-1}, \\
t_{m,k}'(\dlt) &=\sum_{i=1}^k w_{m,i}(\dlt) \ \ \text{and} \ \ s_{m,k}(\dlt) = \sum_{i=1}^k v_{m,k}(\dlt),
\end{align*}
then in the same way as for $(a_{m,i}(\dlt))_i$, by Lemma \ref{lem:onedep} it follows that $(w_{m,i}(\dlt))_i$ are i.i.d. Notice that
\begin{align}\label{eq:t=t'+s}
t_{m,k n_m'(\dlt)} = t_{m,k}'(\dlt)+s_{m,k}(\dlt)
\end{align}
and for $\bullet \in \{\le,\,>\}$ let
\[\mu_m(\dlt) = E[w_{m,1}(\dlt)], \ \mu_m^\bullet(\dlt) = E[w_{m,1}(\dlt) \mid N_m \bullet n_m(\dlt)].\]
To obtain the exponential limit for $\tau(A_m)/E[\tau(A_m)]$ we will show that if $X_m(0)$ has distribution $\nu_m$ and $\dlt_m\to 0$ slowly as $m\to\infty$ then 
\begin{align}\label{eq:propexpsuffcond}
\frac{\dlt_m w_{m,k_m(\dlt_m)}(\dlt_m)}{\mu_m^\le(\dlt_m)} , \ \ \frac{\dlt_m s_{m,k_m(\dlt_m)}(\dlt_m)}{\mu_m(\dlt_m)}  \ \ \text{and} \ \  \frac{ \dlt_m t_{m,k_m(\dlt_m)-1}'(\dlt_m)}{\mu_m^>(\dlt_m)}  - \dlt_m k_m(\dlt_m)\to 0 \ \ 
\text{in $L^1$}.
\end{align}
First we show how \eqref{eq:propexpsuffcond} implies the conclusion, then we prove \eqref{eq:propexpsuffcond}. If $\dlt_m \to 0$ slowly then by assumption (iii) $\mu_m^\le(\dlt_m) = O(\mu_m(\dlt_m))$, and this implies $\mu_m^>(\dlt_m)\sim \mu_m(\dlt_m)$. To see the latter, first we condition:
\[\mu_m(\dlt_m) = \mu_m^>(\dlt_m)(1-q_m(\dlt_m)) + \mu_m^\le(\dlt_m)q_m(\dlt_m),\]
then using $\mu_m^\le(\dlt_m) = O(\mu_m(\dlt_m))$,
\[|\mu_m(\dlt_m) - \mu_m^>(\dlt_m) | \le q_m(\dlt_m)(\mu_m^>(\dlt_m) + O(\mu_m(\dlt_m)))\]
and since $q_m(\dlt_m)\le \dlt_m+p_m=o(1)$, the claim follows. From these observations, $\mu_m^\le(\dlt_m)$ and $\mu_m(\dlt_m)$ can be replaced with $\mu_m^>(\dlt_m)$ in \eqref{eq:propexpsuffcond}, then from \eqref{eq:t=t'+s}, $t_{m,k_m(\dlt_m)-1}(\dlt_m)$ can be replaced with $t_{m,(k-1)n_m'(\dlt)}$ in the third expression in \eqref{eq:propexpsuffcond}. Equation \eqref{eq:tAmsqz} implies that
\[P(0 \le \tau(A_m) - t_{m,(k_m(\dlt_m)-1)n_m'(\dlt)}  \le w_{m,k_m(\dlt_m)} )\to 1,\]
and combining with the adjusted form of the first and third expressions in \eqref{eq:propexpsuffcond} it follows that
\[\dlt_m \tau(A_m)/\mu_m^\le(\dlt_m)  - \dlt_m k_m(\dlt_m) \to 0 \ \ \text{in probability},\]
and from the convergence in distribution of $\dlt_m k_m(\dlt_m)$ to exponential$(1)$,
\begin{align}\label{eq:tAexp1}
P(\dlt_m\tau(A_m)/\mu_m^\le(\dlt_m) > t) \to e^{-t} \ \ \text{as} \ \ m\to\infty.
\end{align}
Now, $\dlt_mk_m(\dlt_m)$ has a uniform exponential tail: $P(\dlt_m k_m(\dlt_m) > a =\dlt_m k ) \le (1-(\dlt_m-p_m))^{a/\dlt_m} \le e^{-a(1-p_m/\dlt_m)}$
which is at most $e^{-a/2}$ if $\dlt_m\to 0$ slowly enough; in particular, $\dlt_m k_m(\dlt_m)$ is uniformly integrable (UI). Since
\[t_{m,k_m(\dlt)n_m'(\dlt)-1} = t_{m,k_m(\dlt)-1}'(\dlt_m) + w_{m,k_m(\dlt_m)}(\dlt_m) + s_{m,k_m(\dlt)-1}(\dlt_m),\]
combining this with the adjusted form of \eqref{eq:propexpsuffcond} it follows that $\dlt_m t_{m,k_m(\dlt)n_m'(\dlt)-1}/\mu_m^\le(\dlt_m)$ is UI, then since $P(\tau(A_m) < t_{m,k_m(\dlt_m)n_m'(\dlt_m)-1})=1$ it follows that $\dlt_m \tau(A_m)/\mu_m^\le(\dlt_m)$ is UI. In turn, 
\[\frac{\dlt_m \tau(A_m)}{\mu_m^\le(\dlt_m)}  - \dlt_m k_m(\dlt_m)\]
is UI, and since it converges to $0$ in probability, it converges to $0$ in $L^1$. Since $\dlt_m k_m(\dlt_m)$ is UI and converges in distribution to exponential$(1)$, $E[\dlt_m k_m(\dlt_m)] \to 1$ and it follows that
\[\frac{\dlt_m}{\mu_m^\le(\dlt_m)} E[\tau(A_m)] \to 1,\]
i.e., $E[\tau(A_m)] \sim \mu_m^\le(\dlt_m)/\dlt_m$. Combining with \eqref{eq:tAexp1} then completes the proof. It remains to prove \eqref{eq:propexpsuffcond}; we begin with the first two statements, which are the easiest. For the first statement, from the i.i.d.~property of $(w_{m,i}(\dlt))_i$ it follows that
\[E_{\nu_m}[w_{m,k_m(\dlt_m)}(\dlt_m) \mid k_m(\dlt_m)] = \mu_m^\le (\dlt_m),\]
so the same holds without conditioning and since $\dlt_m\to 0$ the first statement in \eqref{eq:propexpsuffcond} follows. For the second statement, write
\[E_{\nu_m}[s_{m,k_m(\dlt_m)-1}(\dlt_m)] = \sum_{i\ge 1} E_{\nu_m}[v_{m,i}(\dlt_m) \mid k_m(\dlt_m)>i]P_{\nu_m}(k_m(\dlt_m)>i).\]
With $\phi_m=E_{\nu_m}[\eta_m]$, by definition of regeneration time and the strong Markov property, $E_{\nu_m}[t_{m,n}-t_{m,n-1}]=1$ for all $n$, so $E_{\nu_m}[v_{m,i}(\dlt)] = \phi_m$ for all $m,i,\dlt$ and $\mu_m(\dlt_m) = n_m(\dlt_m)\phi_m$ for all $m,\dlt$. Since $v_{m,i}(\dlt_m)$ is determined by $Y_{m,n_m'(\dlt_m)}$ and since $\{k_m(\dlt_m)>i\}=\{a_{m,j}(\dlt_m)=0 \ \forall j\le i\}$ with each $a_{m,j}(\dlt_m)$ determined by $(Y_{m,n}(\dlt_m)\colon n \in I_{m,j}(\dlt_m))$, using Lemma \ref{lem:onedep} with separated intervals $I_{m,1},\dots,I_{m,i-1}, I_{m,i}\cup \{n_m'(\dlt_m)\}$,
\begin{align*}
E_{\nu_m}[v_{m,i}(\dlt_m) \mid k_m(\dlt_m)>i] &= E_{\nu_m}[v_{m,i}(\dlt_m) \mid a_{m,i}(\dlt_m)=0] \\
&\le \frac{E_{\nu_m}[v_{m,i}(\dlt_m)]}{P_{\nu_m}(a_{m,i}(\dlt_m)=0)} = \frac{\phi_m}{1-q_m(\dlt_m)}.
\end{align*}
If $\dlt_m\to 0$ then $q_m(\dlt_m)\to 0$ and combining the last two displays,
\[E_{\nu_m}[s_{m,k_m(\dlt_m)-1}(\dlt_m)] \sim \phi_m\sum_{i\ge 1} P_{\nu_m}(k_m(\dlt_m)>i) = \phi_m E_{\nu_m}[k_m(\dlt_m)-1].\]
From a union bound, $q_m(\dlt) \le p_m n_m(\dlt)$ for any $m,\dlt$ and since $q_m(\dlt)\ge \dlt$, $n_m(\dlt) \ge \dlt/p_m$ so if $\dlt_m\to 0$ slowly then $n_m(\dlt)\to\infty$. From $\mu_m(\dlt_m)=n_m(\dlt_m)\phi_m$ it follows that $\phi_m= o(\mu_m(\dlt_m))$. Since $k_m(\dlt)$ is geometric$(q_m(\dlt))$ and $q_m(\dlt)\ge \dlt$, $(E_{\nu_m}[\dlt_mk_m(\dlt_m)])_m$ is bounded. Multiplying the above display by $\dlt_m/\mu_m(\dlt_m)$ and using these observations gives the second statement in \eqref{eq:propexpsuffcond}.\\

It remains to prove the third statement in \eqref{eq:propexpsuffcond}. Let $k_m'(\dlt_m):= k_m(\dlt_m)-1$, then conditioned on $k_m'(\dlt_m)$, $y_{m,i} := w_{m,i}(\dlt_m)/\mu_m^>(\dlt_m), \ i = 1,\dots,k_m'(\dlt_m)$ are i.i.d.~with the distribution of $w_{m,1}(\dlt_m)/\mu_m^>(\dlt_m)$ conditioned on $N_m\ge n_m(\dlt_m)$. Let $\tilde E[\cdots] = E_{\nu_m}[\cdots \mid k_m'(\dlt_m)=k_m]$. Condition on $k_m'(\dlt_m)=k_m$ and for each $m$ use the formula from Lemma \ref{lem:gut} on $(y_{m,i})_{i=1}^{k_m}$; since $(y_{m,i})$ are i.i.d.~with mean $1$, in the notation of Lemma \ref{lem:gut} $\mu_i=1$ for every $i$, so 
\begin{align*}
\left|\tilde E\left[\sum_{i=1}^{k_m} y_{m,i} \right] - k_m\right| &\le C k_m^{1/2}a + C k_m \left( k_m^{-1} \tilde E\left[ \sum_{i=1}^{k_m}|y_{m,i} - 1 | \1(|y_{m,i} - 1| \ge a)\right]\right) \\
&\le C k_m^{1/2}a + C k_m \tilde E\left [ \  y_{m,1} \1(y_{m,1} \ge a) \ \right ].
\end{align*}
Assumption (ii) says that, without conditioning on $k_m'(\dlt_m)$, $(y_{m,1})_m$ are UI. Since $N_m\ge n_m(\dlt_m)$ has probability $1-q_m(\dlt_m)=1-o(1)$, by Lemma \ref{lem:UIcond}, $(y_{m,1})_m$ conditioned on $N_m\ge n_m(\dlt_m)$ are UI; in particular $\ep(a) := \sup_m \tilde E\left [ \  y_{m,1} \1(y_{m,1} \ge a) \ \right ]\to 0$ as $a\to\infty$. Dividing by $k_m$,
\[\left|\tilde E \left[\, k_m^{-1}\sum_{i=1}^{k_m} y_{m,i}\right] - 1\right| \le C (k_m^{-1/2}a + \ep(a)).\]
The third expression in \eqref{eq:propexpsuffcond} can be written as
\begin{align*}
\frac{\dlt_m t_{m,(k_m(\dlt_m)-1)}'(\dlt_m)}{\mu_m^>(\dlt_m)} -\dlt_mk_m(\dlt_m)
= \dlt_m k_m'(\dlt_m)  \left(k_m'(\dlt_m)^{-1} \sum_{i=1}^{k_m'(\dlt_m)} \frac{w_{m,i}(\dlt_m)}{\mu_m^>(\dlt_m)} - 1\right) -\dlt_m.
\end{align*}
Since $\dlt_m\to 0$ as $m\to\infty$ we can ignore $-\dlt_m$. Treat $k_m'(\dlt_m)$ once again as a random variable, and for tidiness suppress the $(\dlt_m)$ argument from $k_m',\mu_m^>,w_{m,i}$ and $q_m$. Notating $E[ \, | \cdot| \, ]$ by $E|\cdot|$ for tidiness, the first term has
\begin{align}\label{eq:propexp1}
&E_{\nu_m}\left | \dlt_m k_m'  ((k_m')^{-1} \sum_{i=1}^{k_m'} \frac{w_{m,i}}{\mu_m^>} - 1)\right | \nonumber  \\
&= \dlt_m \sum_{k\ge 1}  k E_{\nu_m}\left | k^{-1} (\sum_{i=1}^k \frac{w_{m,i}}{\mu_m^>} -1)\mid k_m'=k\right | P(k_m'=k) \nonumber \\
&\le \dlt_m  \sum_{k\ge 1} k  C(k^{-1/2} a + \ep(a)) q_m ( 1-q_m)^k \nonumber \\
&\le \dlt_m  C \sum_{k\ge 1} q_mk e^{-q_mk} (k^{-1/2}a + \ep(a)).
\end{align}
We first bound the term with $k^{-1/2}a$, then the one with $\ep(a)$. Since $\dlt_m\sim q_m$, for each $a>0$
\[  C a \, q_m^{1/2}\sum_{k\ge 1} \dlt_m (q_mk)^{1/2} e^{-q_mk} = O\left(\sqrt{q_m} \sum_{k\ge 1} q_m (q_mk)^{1/2} e^{-q_mk}\right).\]
Approximating the sum by an integral,
\[ \sum_{k\ge 1} q_m (q_mk)^{1/2} e^{-q_mk} = \int_0^\infty \sqrt{x}e^{-x}dx + O(q_m)=O(1).\]
so the term with $k^{-1/2}a$ is $O(\sqrt{q_m})\to 0$ as $m\to\infty$, for each $a$. The term with $\ep(a)$ is
\[C \ep(a) \sum_{k\ge 1} \dlt_m q_m k e^{-q_m k} = O\left(\ep(a) \sum_{k\ge 1} q_m (q_mk) e^{-q_mk}\right), \]
and
\[\sum_{k\ge 1} q_m (q_mk) e^{-q_mk} = \int_0^\infty xe^{-x}dx + O(q_m),\]
so the term with $\ep(a)$ is $O(\ep(a))$; letting $a=a_m\to\infty$ slowly as $m\to\infty$ it $\to 0$.
\end{proof}

\subsection{Scaling of expected hitting time}\label{sec:thry3}

As in the beginning of Section \ref{sec:theory}, let $X=(X(t))_{t\ge 0}$ be a pure jump strong Markov process on $S$ with transition rate kernel $q(x,dy)$ and stationary distribution $\pi$. For $A,B\subset S$ define the sequence of transition times from $A$ into $B$ by $\tau_0(A,B)=0$ and
\[\tau_n(A,B) = \inf\{t>\tau_{n-1}(A,B) \colon X(t^-) \in A, \ X(t) \in B\}.\]
Define the stationary transition rate from $A$ into $B$ by
\[q_\pi(A,B) := \int_{x \in A}\pi(dx)q(x,B).\]
The sequence of exit times from $A$ is $(\tau_n(A,A^c))_{n\ge 1}$, and we'll refer to the corresponding trajectory segments, in the sense of Definition \ref{def:renseq}, as \emph{exit cycles}. Let $\tau_\exit(A)=\tau_1(A,A^c)$ denote the first exit time and $q_\exit(A)=q_\pi(A,A^c)$ the stationary exit rate. Let $Y$ be the process of exit states, i.e. $Y_n := X(\tau_n(A,A^c))_{n\ge 1}$, and if $q_\pi(A,A^c)>0$ define the distribution $\mu_A$ supported on $A^c$ by
\begin{align}\label{eq:muAdef}
\mu_A(B)= \frac{q_\pi(A,B\cap A^c)}{q_\pi(A,A^c)}.
\end{align}
The first of this subsection's two main results justifies the title \emph{stationary exit distribution} for \eqref{eq:muAdef} and gives sufficient conditions for the duration of a stationary exit cycle to be equal to the reciprocal of the stationary exit rate.

\begin{proposition}\label{prop:mustat}
Let $X,q,\pi$ and other notation be as above. Suppose that for some $t_0>0$ the measure $P_\pi$ of $X$ with initial distribution $\pi$ is ergodic with respect to the $t_0$ time shift $X\mapsto F_{t_0} X$, with $F_{t_0}$ as in Definition \ref{def:timeshift}. If $q_\exit(A)>0$ then the distribution $\mu_A$ is stationary for $(Y_n)$, and if $(\tau_n(A,A^c)/n)_n$ is uniformly integrable when $X(0)$ has distribution $\mu_A$ then
\[E_{\mu_A}[\tau_\exit(A)] = 1/q_\exit(A).\]
\end{proposition}
The hypothesis of uniform integrability may not be necessary but it eases the burden of proof. To prove Proposition \ref{prop:mustat} we'll need the following continuity result. Say that $\nu_n \to \nu$ setwise if $\nu_n(B) \to \nu(B)$ as $n\to\infty$ for every measurable $B$.\\

\begin{lemma}\label{lem:pushcont}
Let $p(x,dy)$ be a transition kernel on $S$, denoting by $\Phi$ the pushforward operator on measures defined by
\[(\Phi \nu)(B) = \int_{x\in S} \nu(dx)p(x,B).\]
If $\nu_n\to \nu$ setwise then $\Phi \nu_n \to \Phi \nu$ setwise. 
\end{lemma}

\begin{proof}
For measurable $B$ and integer $k\ge 1$, $1\le j\le 2^k$ let
\[A_{jk}:= \{x \in S\colon (j-1)2^{-k}< p(x,B) \le j2^{-k}\},\]
noting that for each $k$, $(A_{jk})_{j=1}^{2^k}$ partitions $S$. Define $\Phi_k$ by
\[(\Phi_k \mu)(B) = \sum_{j=1}^{2^k} \mu(A_{jk})j2^{-k} = \sum_{j=1}^{2^k} \int_{x \in A_{jk}} \mu(dx)j2^{-k}.\]
Then for any $\mu$,
\[(\Phi \mu - \Phi_k\mu)(B) = \sum_{j=1}^{2^k} \int_{x \in A_{jk}} \mu(dx)(p(x,B)-j2^{-k}),\]
and since $|p(\cdot,B)-j2^{-k}|\le 2^{-k}$ on $A_{jk}$,  $|(\Phi \mu - \Phi_k \mu)(B)| \le 2^{-k}$. For each $j,k$, by assumption $\nu_n(A_{jk}) \to \nu(A_{jk})$ as $n\to\infty$, so $(\Phi_k\nu_n)(B) \to (\Phi \nu)(B)$ as $n\to\infty$. Noting
\begin{align*}
|(\Phi \nu - \Phi \nu_n)(B)| &\le |(\Phi \nu - \Phi_k \nu)(B)| \\
&+ |(\Phi_k \nu - \Phi_k \nu_n)(B))| + |(\Phi_k \nu_n - \Phi \nu_n)(B)|
\end{align*}
and taking $\limsup_n$ we find that
\[\limsup_n |(\Phi \nu - \Phi \nu_n)(B)|  \le 2^{-k} + 0 + 2^{-k} = 2^{-(k-1)}.\]
Since $k$ is arbitrary, $(\Phi \nu_n)(B) \to (\Phi \nu)(B)$ as $n\to\infty$.
\end{proof}

\begin{proof}[Proof of Proposition \ref{prop:mustat}]
We'll obtain $\mu_A$ as the setwise limit of the empirical distribution $(\mu_n)$ of the exit process $(Y_n)$, defined by
\[\mu_n(B) = n^{-1}\#\{1\le k\le n\colon Y_n \in B\}.\]
We need to show that
\enumrom
\item $\mu_n \to \mu_A$ setwise,
\item $\mu_A$ is stationary, and
\item $E_{\mu_A}[\tau_\exit(A)] = 1/q_\exit(A)$.
\enumend
We begin with (i); since $\mu_n(A)=0$ it suffices to show $\mu_n(B)\to \mu_A(B)$ as $n\to\infty$ for $B\subset A^c$. Define $N_{A,B}(t) = \#\{s\le t\colon X(s^-) \in A, \ X(s) \in B\}$, the number of jumps from $A$ into $B$ up to time $t$, and note that the empirical distribution $\mu_n$ can be written as
\[\mu_n(B) = n^{-1}N_{A,B}(\tau_n(A,A^c)).\]
We prove (i) in three steps:
\enumalph
\item show that
\begin{align}\label{eq:Nrent}
t^{-1}N_{A,B}(t) \to q_\pi(A,B) \ \ \text{almost surely as} \ \ t\to\infty,
\end{align}
\item show that if $q_\pi(A,B)>0$ then $n^{-1}\tau_n(A,B) \to 1/q_\pi(A,B)$ a.s.~as $n\to\infty$.
\item use $N_{A,B}(\tau_n(A,A^c)) \approx N( q_\pi(A,A^c)^{-1}n)$ to show that
\begin{align}
n^{-1}N_{A,B}(\tau_n(A,A^c)) \to \frac{q_\pi(A,B)}{q_\pi(A,A^c)},
\end{align}
as $n\to\infty$ for $B\subset A^c$, which implies statement (i). 
\enumend
We begin with (a). Now,
\[N_{A,B}(t) - \int_0^t \1(X(s) \in A)q(X(s),B)ds\]
is a martingale, so for any $t$
\[E_\pi [N_{A,B}(t)] = t \int_A \pi(dx)q(x,B) = t\,q_\pi(A,B),\]
which follows from stationarity of $\pi$ and an application of Fubini's theorem. Ergodicity of the $t_0$ time shift then implies that, as $k\to\infty$
\begin{align}\label{eq:Nren}
k^{-1} N_{A,B}(kt_0) \to t_0\, q_\pi(A,B).
\end{align}
With $k_t:= \lf t/t_0\rf $, using monotonicity of $N_{A,B}$,
\begin{align}
\frac{k_t t_0}{t} \frac{N_{A,B}(k_tt_0)}{k_tt_0} \le \frac{N_{A,B}(t)}{t} < \frac{(k_t+1)t_0}{t} \frac{N_{A,B}((k_t+1)t_0)}{(k_t+1)t_0}.
\end{align}
From \eqref{eq:Nren} and the fact that $k_tt_0/t, (k_t+1)t_0/t \to 1$ as $t\to\infty$, we then obtain \eqref{eq:Nrent}, which is (a). Next we show (b); for tidiness use $\tau_n$ and $N$ to denote $\tau_n(A,B)$ and $N_{A,B}$ respectively. Let $k_n$ be such that $k_nt_0 \le \tau_n< (k_n+1)t_0$. Since $n=N(\tau_n)$,
\begin{align}\label{eq:tauexsand}
\frac{N(k_nt_0)}{N(\tau_n)} \frac{k_nt_0}{N(k_nt_0)} \le \frac{\tau_n}{n} < \frac{(k_n+1)t_0}{N((k_n+1)t_0)} \frac{N((k_n+1)t_0)}{N(\tau_n)}.
\end{align}
 Since $t\mapsto N(t)$ is non-decreasing,
 \[\frac{N(k_nt_0)}{N(\tau_n)} \ge \frac{N(k_nt_0)}{N((k_n+1)t_0)} \quad \text{and} \quad  \frac{N((k_n+1)t_0)}{N(\tau_n)} \le \frac{N((k_n+1)t_0)}{N(k_nt_0)}.\]
 Taking $k\to\infty$ in \eqref{eq:tauexsand} and using \eqref{eq:Nren}, (b) will be proved if we can show that $(N((k+1)t_0)-N(kt_0))/N(kt_0) \to 0$ almost surely as $k\to\infty$; since $q_\pi(A,B)>0$ by assumption, because of \eqref{eq:Nren}, it suffices to show that $k^{-1}(N(kt_0)-N((k-1)t_0) \to 0$ a.s.~as $k\to\infty$. If $X(0)$ has distribution $\pi$ then by stationarity $(N(kt_0)-N((k-1)t_0))_{k\ge 1}$ are identically distributed and $E_\pi[N(t_0)]<\infty$, so for any $a>0$
\[\sum_{k\ge 1} P(N(kt_0)-N((k-1)t_0) \ge ka) \le E_\pi[N(t_0)]/a < \infty,\]
and using the Borel-Cantelli lemma, $\limsup_k k^{-1}(N(kt_0)-N((k-1)t_0))\le a$ a.s.~for every $a>0$. Since $N(k(t_0))\ge N((k-1)t_0)$ it follows that
\[\lim_{k\to\infty} k^{-1}(N(kt_0)-N((k-1)t_0))=0\]
a.s., which completes the proof of (b). For (c), use $N,q$ to denote $N_{A,B},q_\pi(A,A^c)$ respectively and let $\ep>0$, then by (a), a.s.~for large $n$
\[(q^{-1}-\ep)\frac{N((q^{-1}-\ep)n)}{(q^{-1}-\ep)n} \le \frac{N(\tau_n(A,A^c))}{n} \le (q^{-1}+\ep)\frac{N((q^{-1} + \ep)n)}{(q^{-1}+\ep)n}.\]
Using \eqref{eq:Nrent} and the fact that $\ep>0$ is arbitrary, (c) is obtained.\\

From (i), obtain (ii) as follows: let $p(x,dy)$ be the transition kernel of $(Y_n)$ and $\Phi$ the pushforward as in Lemma \ref{lem:pushcont}. Since $\mu_{n+1}=\Phi \mu_n$, taking $n\to\infty$ and using Lemma \ref{lem:pushcont},
\[\mu_A=\lim_{n\to\infty} \mu_{n+1} = \lim_{n\to\infty} \Phi\mu_n = \Phi \lim_{n\to\infty} \mu_n = \Phi \mu_A.\]
Then, (iii) is obtained as follows: by stationarity, 
\[E_{\mu_A}[\tau_n(A,A^c)/n] = E_{\mu_A}[\tau_\exit(A)].\]
Since $\tau_n(A,A^c)/n$ converges almost surely to $1/q_\pi(A,A^c)=1/q_\exit(A)$ and is uniformly integrable by assumption, it converges in $L^1$, giving
\[1/q_\exit(A) =\lim_{n\to\infty} E_{\mu_A} [\tau_n(A,A^c)/n]=E_{\mu_A}[\tau_\exit(A)].\]
\end{proof}

The second of this subsection's two main results gives sufficient conditions for the hitting time of the rare event, initialized in the regeneration measure, to scale in the same way as the exit time from the rare event, initialized in the stationary exit distribution. Condition (i) says that from the stationary exit distribution there are no early returns to the rare event in the sense that regeneration precedes hitting $A_m$ with high probability. Condition (ii) says regeneration occurs early over the course of a stationary exit cycle. Condition (iii) prescribes that $A_m$ is a rare event.

\begin{proposition}\label{prop:tauex}
Let $(X_m)_{m\in \N}$, $(\eta_m)$, $(\nu_m)$ and $(A_m)$ be as in Proposition \ref{prop:exp}, and suppose that for each $m$, $(X_m)$ has a stationary distribution $\pi_m$ such that for some $t_0>0$, the distribution $P_{\pi_m}$ of $X_m$ with initial distribution $\pi_m$ is ergodic with respect to the $t_0$ time shift $X\mapsto F_{t_0} X$, with $F_{t_0}$ as in Definition \ref{def:timeshift}. Let $\mu_m$ denote the stationary exit distribution for $A_m$ defined by \eqref{eq:muAdef} and $\tau(A_m)$ the hitting time of $A_m$. Let $\tau_\exit$ be as defined above \eqref{eq:muAdef}.  Suppose that as $m\to\infty$,
\enumrom
\item $P_{\mu_m}(\eta_m < \tau(A_m))\to  1$,
\item $E_{\mu_m}[\eta_m] = o(E_{\mu_m}[\tau_\exit(A_m)])$ and 
\item $\pi(A_m)\to 0$.
\enumend
Then
\[E_{\nu_m}[\tau(A_m)] \sim E_{\mu_m}[\tau_\exit(A_m)] \ \ \text{as} \ \ m\to\infty.\]
\end{proposition}

\begin{proof}
For tidiness we suppress the index $m$ and write $\mu_m,\nu_m$ and $\mu,\nu$ and $\tau_\exit(A_m),\tau(A_m)$ as $\tau_\exit,\tau_A$. We break up the next exit interval into three parts: from last exit to regeneration, regeneration to hitting $A$, and hitting $A$ to exiting $A$; formally, we write
\begin{align*}
    |E_{\mu}[\tau_\exit] - E_{\nu}[\tau]| &\le |E_\mu[\tau_\exit] - E_\mu[\tau_A]| \\
    &+ |E_\mu[\tau_A] - E_\mu[(\tau_A-\eta)_+]| \\
    &+ |E_\mu[(\tau_A-\eta)_+] - E_\nu[\tau_A]|,
\end{align*}
and need to show each of the three terms on the right-hand side is either $o(E_{\nu}[\tau_A])$ or $o(E_{\mu}[\tau_\exit])$. First, since the next exit from $A$ occurs after the next visit to $A$, $\tau_\exit > \tau_A$ and
\[|E_\mu[\tau_\exit] - E_\mu[\tau_A]| = E_\mu[\tau_\exit-\tau_A],\]
the latter being the expected time spent in $A$ over one stationary exit cycle. Since, by assumption, $P_{\pi}$ is ergodic for $X$, applying Lemma \ref{lem:excprinc} to the stopping time $\tau_\exit$ and distribution $\mu$ that, by Proposition \ref{prop:mustat}, has the required property,
    \[\pi(A) = \frac{E_\mu[T_A(\tau_\exit)]}{E_\mu[\tau_\exit]}=\frac{E_\mu[\tau_\exit-\tau_A]}{E_\mu[\tau_\exit]}.\]
Since $\pi(A_m)\to 0$ as $m\to\infty$, $E_\mu[\tau_\exit-\tau_A] = o(E_\mu[\tau_\exit])$. Second, 
\[|E_{\mu}[\tau_A] - E_{\mu}[(\tau_A-\eta)_+]| \le E_{\mu}[\eta]\]
which is $o(E_{\mu}[\tau_\exit])$ by assumption. Third,
\[E_{\mu}[(\tau_A-\eta)_+] = P_{\mu}(\eta < \tau_\exit)E_\nu[\tau_A],\]
so
\[|E_{\mu}[(\tau_A-\eta)_+] - E_{\nu}[\tau_A] | \le (1-P_{\mu}(\eta<\tau_A)) E_{\nu}[\tau_A]\]
which is $o(E_{\nu}[\tau_A])$ since by assumption $P_{\mu}(\eta<\tau_A) \to 1$.
\end{proof}

\section{Regeneration time for $(X_\lbd)$}\label{sec:renew}
In order to prove Theorem \ref{thm:main} we define a regeneration time $\eta_\lambda$ for the dynamic Ppp$(\lbd)$, that we call the \emph{refresh time}, using a two-phase `flush and pin' approach. The flush phase consists in waiting until all initial particles have vanished, then the pin phase consists in waiting until the number of particles is equal to $\lfloor \lambda\rfloor$. Formally, let $\tau_{\flsh,\lambda}=\inf\{t\colon X_\lambda(t)\cap X_\lambda(0) = \emptyset\}$ then let $\tau_{\pin,\lambda}(t_0) = \inf\{t>0\colon N_\lambda(t+t_0)= \lf \lambda 
\rf\}$ where $N_\lambda(t)=|X_\lambda(t)|$ is the number of particles. Then the refresh time is defined as follows.

\begin{definition}[Refresh time]\label{def:refresh}
With $\tau_{\flsh,\lbd},\tau_{\pin,\lbd}$ as above, the \emph{refresh time} for the dynamic Poisson point process $X_\lbd$ is 
\[\eta_\lambda:=\tau_{\pin,\lbd}(\tau_{\flsh,\lbd})+\tau_{\flsh,\lbd}.\]
\end{definition}

Statement (b) of the next lemma ensures that $\eta_\lbd$ is a regeneration time in the sense of Definition \ref{def:renew} and gives the corresponding regeneration measure. Statement (c) asserts a slightly stronger property: that the stopping sequence $(t_{\lbd,n}^\ren)$ corresponding to $\eta_\lbd$ (see Definition \ref{def:renew}) has i.i.d.~waiting times. Statement (a) is needed for some later estimates.

\begin{lemma}\label{lem:renprop}
Let $\nu_\lbd$ be the distribution of $\{u_i\}_{i=1}^{\lf \lbd\rf}$ where $(u_i)$ are i.i.d.~uniform on $[0,1]$. For any distribution of $X_\lambda(0)$,
\enumalph
\item conditioned on $\tau_{\flsh,\lbd}$, for $t\ge \tau_\flsh$, $X_\lbd(t)$ has the distribution of a Poisson point process on $[0,1]$ with intensity $\lambda (1-e^{-t})$, and
\item $X_\lambda(\eta_\lambda)$ has distribution $\nu_\lbd$ and is independent of $X_\lambda(0)$,
\item the waiting times $w_{\lbd,n}^\ren := t_{\lbd,n}^\ren - t_{\lbd,n-1}^\ren$ for the stopping sequence $(t_{\lbd,n}^\ren)$ (see paragraph above) are i.i.d.
\enumend

\end{lemma}
\begin{proof}
Construct $X_\lbd(t)$ from $X_\lbd(0)$ together with i.i.d.~sequences $(r_i,w_i,u_i,s_i)$, where $r_i,w_i,s_i$ are exponential with mean $1$ and $u_i$ are uniform on $[0,1]$, as follows. Label $X_\lbd(0)=\bigcup_{i=1}^{N_\lbd(0)} \{x_i\}$. Initial points that are still present at time $t$ are $V_\lbd(t):= \bigcup\{x_i\colon i\le N_\lbd(0), \ r_i>t\}$. The arrival time of the $i^\text{th}$ point after time $0$ is $t_i:=\sum_{j\le i}w_j/\lbd$. Points that appear after time $0$ and are present at time $t$ are $W_\lbd(t) = \bigcup_{i\in K_\lbd(t)}\{u_i\}$ where $K_\lbd(t):=\{i\colon t_i \le t < t_i+s_i\}$. Then, let $X_\lbd(t)=V_\lbd(t)\cup W_\lbd(t)$.\\

For each $t$, $W_\lbd(t)$ is a Poisson point process on $[0,1]$ with intensity $\lbd(1-e^{-t})$. To see this, first note that a point that arrives at time $s\le t$ is still present at time $t$ with probability $e^{-(t-s)}$. Multiplying by the arrival rate $\lbd$ and integrating over $s\in [0,t]$, Poisson thinning implies that $|K_\lbd(t)|$ is Poisson distributed with mean $\lbd \int_0^t e^{-(t-s)}ds = \lbd(1-e^{-t})$. Since $K_\lbd$ is determined by $(w_i),(s_i)$ it is independent of $(u_i)$, so conditioned on $K_\lbd(t)$, $\{u_i\}_{i \in K_\lbd(t)}$ has the same distribution as $|K_\lbd(t)|$ i.i.d.~uniform points on $[0,1]$, and the claim follows.\\

Since $V_\lbd,W_\lbd$ are determined by $(r_i)$, $(w_i,u_i,s_i)$ respectively, they are independent. Moreover, $\tau_{\flsh,\lbd}=\inf\{t\colon V_\lbd(t)=\emptyset\}$ is determined by $V_\lbd$, and $X_\lbd(t)=W_\lbd(t)$ for $t\ge\tau_\flsh$. Combining this information with the distribution of $W_\lbd(t)$, (a) is obtained. To obtain (b), first note that by construction, $X_\lbd(0)$ and $(u_i)$ are independent, and that $\eta_\lbd, K_\lbd$ are independent of $(u_i)$, since $\tau_\flsh$ is determined by $V_\lbd$, which is determined by $N_\lbd(0)$ and $(r_i)$, and conditioned on $\tau_\flsh$, $\tau_\pin(\tau_\flsh)$ is determined by $K_\lbd$, which is determined by $(w_i),(s_i)$. Therefore, conditioned on $\eta_\lbd$, $K_\lbd$ and $X_\lbd(0)$, $\{u_i\}_{i \in K_\lbd(\eta_\lbd)}$ has the same distribution as $|K_\lbd(\eta_\lbd)|=|W_\lbd(\eta_\lbd)|$ i.i.d.~uniform points on $[0,1]$. Since $\eta_\lbd \ge \tau_{\flsh,\lbd}$, $ |W_\lbd(\eta_\lbd)|=|X_\lbd(\eta_\lbd)| =\lf \lbd\rf$, by definition of $\eta_\lbd$, which completes the proof of (b).\\

For (c), by the strong Markov property and the fact that $|X_\lbd(t_{\lbd,n}^\ren)|=\lf\lbd\rf$ it suffices to show that if $P(N_\lbd(0) = \lf\lbd\rf)=1$ then $\eta_\lbd$ is independent of $X_\lbd(0)$. However, $\tau_{\flsh}$ is determined by $N_\lbd(0)$ (which by assumption is deterministic) and $\{r_i\}$, and conditioned on $\tau_\flsh$, $\tau_\pin(\tau_\flsh)$ is determined by $K_\lbd$, which is determined by $(w_i),(s_i)$.
\end{proof}

The construction of the process introduced in the proof of Lemma \ref{lem:renprop} can be used to show that the dynamic Ppp is strong mixing and therefore ergodic, a property that will be needed in order to make use of Lemma \ref{lem:excprinc} and Propositions \ref{prop:mustat} and \ref{prop:tauex}.

\begin{lemma}\label{lem:dynerg}
Let $P_{\pi_\lbd}$ be the measure of the dynamic Ppp on $[0,1]$ with intensity $\lbd$ and initial distribution $\pi_\lbd$, a Poisson point process on $[0,1]$ with intensity $\lbd$. For each $\lbd$ and $t>0$, $P_{\pi_\lbd}$ is strong mixing, and in particular ergodic, with respect to the time shift on trajectories defined by $F_t((x(s))_{s \in \R_+}) = (x(t+s))_{s \in \R_+}$.
\end{lemma}

\begin{proof}
We first show that $\pi_\lbd$ is stationary as this will be needed in the proof that it's strong mixing.\\

\noindent\textit{Stationarity. }In the notation of the proof of Lemma \ref{lem:renprop}, first note that if $V_\lbd(t)$ is a Ppp$(\mu)$ then by the Poisson thinning property $V_\lbd(t+s)$ is a Ppp$(\mu e^{-s})$. Thus if $X_\lbd(0)$ is a Ppp$(\lbd)$ then $V_\lbd(t)$ is a Ppp$(\lbd e^{-t})$. We showed in the proof of Lemma \ref{lem:renprop} that for each $t$, $W_\lbd(t)$ is a Ppp$(\lbd (1-e^{-t}))$. Since $V_\lbd,W_\lbd$ are independent, by the superposition property $X_\lbd(t)=V_\lbd(t)\cup W_\lbd(t)$ is a Ppp$(\lbd e^{-t}+\lbd(1-e^{-t})) = $ Ppp$(\lbd)$, which shows that $\pi_\lbd$ is stationary.\\

\noindent\textit{Strong mixing. }It suffices to check that for each $T>0$, if $B,C$ are events determined by $(X(t))_{t\le T}$ that have $P_{\pi_\lbd}(B),P_{\pi_\lbd}(C)>0$ then
\begin{align}\label{eq:smix}
P_{\pi_\lbd}( F_t^{-1}(C) \mid B) \to  P_{\pi_\lbd}(C),\quad \text{as } t\to\infty.
\end{align}
We use a similar construction as in the proof of Lemma \ref{lem:renprop}. Let $X_\lbd^{(1)}(0),X_\lbd^{(2)}(T)$ be independent with distribution $\pi_\lbd$ and define i.i.d.~sequences $(r_i^{(1)},r_i^{(2)},w_i,u_i,s_i)$, where $r_i^{(1)},r_i^{(2)},w_i,s_i$ are exponential with mean $1$ and $u_i$ are uniform on $[0,1]$, denoting their measure by $Q$. Define $V_\lbd^{(1)},K_\lbd^{(1)},W_\lbd^{(1)},X_\lbd^{(1)}$ from $(r_i^{(1)},w_i,u_i,s_i)$ as in the proof of Lemma \ref{lem:renprop}. In addition, label $X_\lbd^{(2)}(T)=\bigcup_{i=1}^{N_\lbd^{(2)}(T)}\{x_i^{(2)}\}$, let $V_\lbd^{(2)}(T+t)=\bigcup\{x_i^{(2)}\colon i\le N_\lbd^{(2)}(T),\ r_i^{(2)}>t\}$ and let $W_\lbd^{(2)}(t)=\bigcup_{i \in K_\lbd^{(2)}(t)}\{u_i\}$ where $K_\lbd^{(2)}(t):= \{i\colon T < t_i \le t < t_i+s_i\}$ refers to points that appear after time $T$ and are present at time $t$. Let $X_\lbd^{(2)}(t) = V_\lbd^{(2)}(t)\cup W_\lbd^{(2)}(t)$ and note that $(X_\lbd^{(2)}(t))_{t\ge T}$ is independent of $(X_\lbd^{(1)}(t))_{t\le T}$. Let $\tau_1=\inf\{t\colon V_\lbd^{(1)}(T+t)=\emptyset\}$, $\tau_2=\inf\{t\colon K_\lbd^{(1)}(T+t) \cap K_\lbd^{(1)}(T)=\emptyset\}$ and $\tau_3=\inf\{t\colon V_\lbd^{(2)}(T+t)=\emptyset\}$ and let $\tau=\tau_1\vee \tau_2\vee \tau_3$. Then $X_\lbd^{(1)}(t)=X_\lbd^{(2)}(t)$ for all $t\ge T+\tau$ and conditioned on $X_\lbd^{(1)}(T),\, X_\lbd^{(2)}(T)$, $\tau$ is equal to the maximum of $|X_\lbd^{(1)}(T)|+|X_\lbd^{(2)}(T)|$ independent exponential$(1)$ random variables. In particular, since $X_\lbd^{(1)}(T),\, X_\lbd^{(2)}(T)$ are finite $Q$-almost surely, $Q(\tau<\infty)=1$.\\

Let $B_1=\{(X_\lbd^{(1)}(t))_{t\in \R_+} \in B\}$ and for $j\in \{1,2\}$, $t\in \R_+$, let $C_j(t)=\{(X_\lbd^{(j)}(t+s))_{s\in \R_+} \in C\}$. Note that $C_1(t) = \{(X_\lbd^{(1)}(s))_{s \in \R_+} \in F_t^{-1}(C)\}$ and that for each $t$, using the Markov property and stationarity of $\pi_\lbd$, $P_{\pi_\lbd}(C_2(t)) = P_{\pi_\lbd}(C)$. By definition, $Q(C_1(t) \mid B_1) = P_{\pi_\lbd}(F_t^{-1}(C) \mid B)$ and by independence of $(X_\lbd^{(1)}(t))_{t\le T}$ and $(X_\lbd^{(2)}(t))_{t\ge T}$, $Q(C_2(t) \mid B_1) = Q(C_2(t))=P_{\pi_\lbd}(C)$. Since $X_\lbd^{(1)}(t)=X_\lbd^{(2)}(t)$ for $t>T+\tau$,
\[Q(C_1(t) \cap \{T+\tau \le t\} \mid B) = Q(C_2(t) \cap \{T+\tau \le t\}\mid B),\]
so
\[|Q(C_2(t)\mid B_1) - Q(C_1(t) \mid B_1)| \le 2Q(T+\tau > t \mid B).\]
In particular, \eqref{eq:smix} holds if $Q(T+\tau > t\mid B) \to 0$ as $t\to\infty$. We showed above that $Q(\tau<\infty)=1$, so $Q(\tau<\infty \mid B)=1$ and the claim follows.
\end{proof}

In order to prove the next result we will need some estimates for the \emph{dynamic Poisson process with intensity $\rho$}, denoted $N(t)$, with transition rates
\begin{align}\label{eq:dynPois}
N \to \begin{cases} 
N +1 & \text{at rate} \ \rho,\\
N-1  & \text{at rate} \ N.\end{cases}
\end{align}

\begin{lemma}\label{lem:countest}
Let $N$ be as in \eqref{eq:dynPois} and let $\tau(n)=\inf\{t\colon N(t)=n\}$. There are $c,C>0$ such that 
\enumalph
\item $P(\tau(\lfloor \rho\rfloor) >t+\log(E[ \, \rho^{-1/2}|N(0)-\rho|\, ])) \le C e^{-ct}$ and 
\item $P(\tau(0)<\tau(\lf \rho\rf) \mid N(0)\ge 1)\le C/\rho$.
\enumend
Moreover, for each $\alpha\in (0,1)$ there is $c,C>0$ so that for all $\rho$,
\enumalph
\item[(c)] $P(\tau(0)<\tau(\lf \rho\rf) \mid N(0)\ge \alpha\rho)\le C e^{-c\rho}$ and
\item[(d)] $P(\tau(0)\le e^{c\rho} \mid N(0)\ge \alpha\rho)\le Ce^{-c\rho}$.

\enumend
\end{lemma}
Statement (a) makes sense after observing that for large $\lf\rho\rf$, $\rho^{-1/2}(N-\rho)$ resembles an Ornstein-Uhlenbeck diffusion, although it does not follow directly from the invariance principle, since our result holds for all $\rho$ and also handles the tail of the hitting time distribution. Statements (b) and (c) express that $N$ prefers to hit $\lf\rho\rf$ than $0$, and (d) says that from large initial values, $N$ does not hit $0$ for an exponentially long in $\rho$ amount of time, with exponentially small in $\rho$ error probability. The proof of Lemma \ref{lem:countest} can be found in Section \ref{sec:Nest} and is achieved with the help of martingales, Doob's inequality and stopping arguments.\\

Next, we use Lemmas \ref{lem:renprop} and \ref{lem:countest} to establish some properties of the flush, pin and refresh times.

\begin{lemma}\label{lem:renest}
Let $\tau_{\flsh,\lbd},\tau_{\pin,\lbd},\eta_\lbd,\nu_\lbd$ be as in Definition \ref{def:refresh}. For some $c,C>0$ and all $\lbd$,
\enumalph
\item $P(\tau_{\flsh,\lbd}\le t \mid X_\lbd(0)) = (1-e^{-t})^{N_\lbd(0)}\le e^{-N_\lbd(0)e^{-t}}$,\\
$P(\tau_{\flsh,\lbd} > t \mid X_\lbd(0)) = 1 - (1-e^{-t})^{N_\lbd(0)}\le N_\lbd(0)e^{-t}$,
\item $P(\eta_\lbd > \tau_{\flsh,\lbd}\vee (\log(\lbd)/2) + t ) \le C e^{-ct},$
\item $|E_{\nu_\lbd}[\eta_\lbd] - \log\lbd |\le C$,
\item $P( \eta_\lbd > \log(N_\lbd(0)\vee \sqrt{\lbd})+t ) \le Ce^{-ct}$,
\enumend
\end{lemma}

\begin{proof}We prove statements (a)-(d) in that order.
\enumalph
\item $\tau_{\flsh,\lbd}$ is the maximum of $N_\lbd(0)$ independent exponential($1$) random variables, from which the cumulative distribution function of $\tau_{\flsh,\lbd}$ is as given. For the first inequality use $1-x\le e^{-x}$. For the second inequality use $(1-x)^n \ge 1-nx$.
\item First, pass to the second moment, using
\[E[ \, |N(\tau_{\flsh})-\lbd| \, ] \le E[\, (N(\tau_{\flsh}) - \lbd)^2\, ]^{1/2}.\]
By Lemma \ref{lem:renprop} (a), conditioned on $\tau_{\flsh,\lbd}$, for $t\ge \tau_\flsh$ $N_\lbd(t)$ is Poisson distributed with mean $\lbd(1-e^{-t})$. In particular,
\begin{align}\label{eq:almostPoisson}
E[\, (N_\lbd(t) - \lbd)^2\, ] 
&= E[\, N_\lbd(t)^2\, ] - 2\lbd E[\, N_\lbd(t) \, ] + \lbd^2 \nonumber \\
&= \lbd^2 (1-e^{-t})^2 + \lbd (1-e^{-t}) - 2\lbd^2(1-e^{-t}) + \lbd^2 \nonumber \\
&= \lbd +  \lbd e^{-t}(\lbd e^{-t}-1).
\end{align}
If $t\ge \log (\lbd)/2$, then $|\lbd e^{-t}(\lbd e^{-t}-1)| \le \lbd$. Therefore

\[E[ \, |N_\lbd(t \vee (\tau_{\flsh,\lbd}\vee \log(\lbd)/2) - \lbd| \, ] \le \sqrt{2\lbd}.\]

Then use Lemma \ref{lem:countest} (a) with $\rho=\lbd$, initialized at time $\tau_{\flsh}\vee \log(\lbd)/2$ instead of $0$; putting the upper bound $\lbd^{-1/2}\sqrt{2\lbd}=\sqrt{2}$ in place of
\[E[ \, \rho^{-1/2} |N(0)-\rho|\, ],\]
and $(\eta_\lbd-\tau_{\flsh,\lbd}\vee\log(\lbd)/2)_+$ in place of $\tau(\lf\rho\rf)$, (b) is obtained except with $t+\log \sqrt{2}$ in place of $t$; to obtain (b), replace $t+\log \sqrt{2}$ with $t$ and $2^{c/2}C$ with $C$.\\

\item If $X_\lbd(0)$ has distribution $\nu_\lbd$ then $N_\lbd(0)=\lf \lbd \rf$, so $\tau_\flsh$ is the maximum of $\lf \lbd\rf$ independent exponential$(1)$ random variables and
\[E_{\nu_\lbd}[\tau_\flsh] = \sum_{i=1}^{\lf \lbd\rf} 1/i = \log(\lbd)+O(1),\]
where the summation comes from examining the process that gives the number of initial particles that still remain, which is a Markov chain on $\N$ with transitions $n\to n-1$ at rate $n$, for which the expected time to reach $0$ from $n$ is the sum of the expected waiting times $1/n+1/(n-1) + \dots 1/1$ in states $n,n-1,\dots,1$. To obtain (c) it remains to show that $E_{\nu_\lbd}[\,\eta_\lbd-\tau_{\flsh,\lbd} \,]=O(1)$. Noting that
\begin{align*}
0 \le \eta_\lbd - \tau_{\flsh,\lbd} &\le (\eta_\lbd - \tau_{\flsh,\lbd} \vee (\log(\lbd)/2))_+ \\
&+ \1(\tau_{\flsh,\lbd} < \log(\lbd)/2)\log(\lbd)/2,
\end{align*}
and taking expected values,
\begin{align}\label{eq:flshlog}
&| E_{\nu_\lbd}[ \, \eta_\lbd  - \tau_{\flsh,\lbd} \,] - E_{\nu_\lbd} [ \, (\eta_\lbd - \tau_{\flsh,\lbd} \vee (\log(\lbd)/2))_+\,] | \nonumber \\
&\le \frac{\log\lbd}{2} P_{\nu_\lbd}(\tau_{\flsh,\lbd} \le \log(\lbd)/2).
\end{align}
If $X_\lbd(0)$ has distribution $\nu_\lbd$ then $N_\lbd(0)=\lf\lbd\rf$, so using the first part of (a),\\
\[P(\tau_{\flsh,\lbd} \le \log(\lbd)/2) \le e^{-\lf\lbd\rf /\sqrt{\lbd}}\le e^{1 - \sqrt{\lbd}}\]
if $\lbd\ge 1$, and the right-hand side of \eqref{eq:flshlog} is $O( e^{-\sqrt{\lbd}}\log \lbd)=o(1)$. Then, from (b),
\begin{align*}
&E[(\eta_\lbd-\tau_{\flsh,\lbd}\vee(\log(\lbd)/2))_+] \\
&\le \int_0^\infty P(\eta_\lbd > (\tau_{\flsh,\lbd}\vee (\log(\lbd)/2)+t)dt \le C/c=O(1),
\end{align*}
which completes the demonstration of (c).\\

\item Let $e_1 = (\tau_{\flsh,\lbd} - \log N_\lbd(0))_+$ and $e_2 = (\eta_\lbd - \tau_{\flsh,\lbd} \vee \log(\lbd) / 2)_+$, then
\begin{align*}
\eta_\lbd &\le e_2 + \tau_{\flsh,\lbd} \vee \log(\lbd)/2 \\
&= e_2 + (e_1 + \log N_\lbd(0))\vee \log\sqrt{\lbd} \\
&\le e_2 + e_1 + \log (N_\lbd(0)\vee \sqrt{\lbd}).
\end{align*}
Taking $t_0=\log N_\lbd(0)+t$ in the second part of (a) gives $P(e_1 > t) \le e^{-t}$, then using (b), $P(e_2>t \mid e_1)\le Ce^{-ct}$ (the bound does not depend on the value of $\tau_{\flsh,\lbd}$). For $\theta<c\wedge 1$ it follows that $E[e^{\theta (e_1+e_2)}]<\infty$, then
\[P(\eta_\lbd > \log(N_\lbd(0)\vee \sqrt{\lbd})+t) \le P(e_1+e_2>t) \le e^{-\theta t} E[e^{\theta (e_1+e_2)}],\]
and after relabelling constants, this gives (d).
\enumend
\end{proof}

\section{Proof of Theorem \ref{thm:main}}\label{sec:proof}

The proof of Theorem \ref{thm:main} is broken up into two parts. The first part, Proposition \ref{prop:iv}, handles the initial distribution. The second part proves Theorem \ref{thm:main}, assuming the initial distribution is the regeneration measure $\nu_\lambda$.

\begin{proposition}\label{prop:iv}
Let $X_\lambda$ be the dynamic Poisson point process on $[0,1]$ with intensity $\lambda$, with regeneration time $\eta_\lambda$ as defined above and let $N_\lbd = |X_\lbd|$ be its cardinality. Suppose $\log N_\lbd(0)= O(\lbd)$ and $X_\lambda(0)$ is $\alpha$-dense at scale $w_\lambda/2$, for some $\alpha>0$ and every $\lbd$. If $\liminf_\lbd \lbd w_\lbd/\log\lbd$ and $C>0$ are large enough then
\enumar
\item $P(\eta_\lambda \le C\lbd) \to 1 \ \ \text{as} \ \ \lambda\to\infty$ and
\item $P(\eta_\lambda < \tau_{A(w_\lambda)}) \to 1$ as $\lambda\to\infty$.

\enumend
\end{proposition}

\begin{proposition}\label{prop:main}
Let $X_\lambda$ be the dynamic Poisson point process on $[0,1]$ with intensity $\lambda$, with regeneration measure $\nu_\lambda$ as defined above. If $X_\lambda(0)$ has distribution $\nu_\lambda$ and conditions (a)-(b) of Theorem \ref{thm:main} hold (for large enough $C>0$ in (a)) then statements (i) and (ii) of Theorem \ref{thm:main} hold.
\end{proposition}

Theorem \ref{thm:main} is easily proved from the above propositions.
\begin{proof}
    Suppose $X_\lambda(0)$ is $\alpha$-dense at scale $w_\lbd/2$ and $\log |X_\lbd(0)| = O(\lbd)$.  
    By definition of $\eta_\lambda, \nu_\lambda$, $X_\lambda(\eta_\lambda)$ has distribution $\nu_\lambda$, and using 1, 2 of Proposition \ref{prop:iv}, with probability tending to $1$ as $\lbd\to\infty$, 
    \enumar 
    \item $\eta_\lambda = O(\lbd)$ which for $\liminf_\lbd \lbd w_\lbd/\log\lbd =: \gamma > 4$ is $o(E_{\nu_\lbd}[\tau_{A(w_\lbd)}])$, since by statement (ii) of Proposition \ref{prop:main} and $e^x/x\ge e^{x/2}$ for large $x$, $E_{\nu_\lbd}[\tau_{A(w_\lbd)}] \sim e^{\lbd w_\lbd} / (\lbd^2 w_\lbd(1-w_\lbd))\ge e^{\lbd w_\lbd/2}/\lbd = \lbd^{1+\dlt}$ for large $\lbd$ with $\dlt=\gamma/2-2>0$, and
    \item $X_\lambda(t) \notin A(w_\lambda)$ for all $t\le \eta_\lambda$.
    \enumend
    Statement 1 implies the time interval $[0,\eta_\lbd]$ is of vanishing size after rescaling time by $E_{\nu_\lbd}[\tau_{A(w_\lbd)}]$ and statement 2 implies that $A(w_\lbd)$ is unlikely to occur in that time interval. Using the strong Markov property at time $\nu_\lambda$ and Proposition \ref{prop:main}, Theorem \ref{thm:main} follows.
\end{proof}

The rest of this section is devoted to the results needed to prove Propositions \ref{prop:iv} and \ref{prop:main}, and finishes with the proof of Propositions \ref{prop:iv} and \ref{prop:main}. Here is a list of the required results:

\enumar
\item Lemma \ref{lem:alphadense}: showing that a Ppp is $\alpha$-dense (see Definition \ref{def:alphadense}) at the required scale.
\item Lemma \ref{lem:piAw}: estimating the stationary probability $\pi(A(w_\lbd))$ of the gap event and the conditional distribution $\pi_\lbd(\cdot \mid A(w_\lbd))$.
\item Lemma \ref{lem:numudens}: bounding the density function of the regeneration measure $\nu_\lbd$ and the stationary exit-from-$A(w_\lbd)$ distribution $\mu_\lbd$ with respect to $\pi_\lbd$.
\item Lemma \ref{lem:prop1cond}: verifying the conditions of Proposition \ref{prop:exp} for the dynamic Ppp,
\item Lemma \ref{lem:exitest}: verifying the conditions of Proposition \ref{prop:mustat} and estimating the stationary exit rate $q_\exit(A(w_\lambda))$ in terms of $\pi_\lbd$,
\item Lemma \ref{lem:denstime}: lower bound on the hitting time $\tau_{A(w_\lbd)}$ of the gap event from an initial configuration that is $\alpha$-dense at scale $w_\lbd/2$,
\item Lemma \ref{lem:prop3cond}: verifying the conditions of Proposition \ref{prop:tauex} for the dynamic Ppp.
\enumend

\begin{lemma}\label{lem:alphadense}
Let $X$ be a Poisson point process on $[0,1]$ with intensity $\lbd$. For each $\alpha\in (0,1/2)$ there is $C>0$ such that if $\liminf_\lbd \lbd w_\lbd / \log\lbd \ge C$ then
\[P(X \ \text{is} \ \lbd\alpha\text{-dense at scale} \ w_\lbd) \to 1 \ \text{as} \ \lbd\to\infty.\]
\end{lemma}

\begin{proof}
For $k=0,\dots, \lf 2/w\rf-1$ let $I_k=(kw_\lbd/2,(k+1)w_\lbd/2)$, then for each $\alpha<1/2$, by a standard large deviations estimate $P(X(I_k) < \alpha \lbd w_\lbd) \le C_1e^{-c\lbd w_\lbd} $ where $c,C_1>0$ depend on $\alpha$ but not $\lbd$. Taking a union bound, if $\lbd w_\lbd \ge (1/c) \log \lbd$ then
\[P(\min_k X(I_k) < \alpha \lbd w_\lbd) \le (2C/w_\lbd)e^{-c\lbd w_\lbd}\le \frac{2C\lbd}{a\log\lbd}\lbd^{-1}=o(1).\]
Since each $I_k$ has length $w_\lbd/2$ and $I_1\subset (0,w)$ while $I_{\lf 2/w\rf-1} \subset (1-w,1)$, for any $y\in [0,1-w]$, $(y,y+w) \supset I_k$ for some $k$, so
\[P(X \ \text{is} \ \alpha\text{-dense at scale} \ w_\lbd) \ge 1 - P(\min_k X(I_k) < \alpha \lbd w_\lbd) \to 1 \ \text{as} \ \lbd\to\infty.\]
\end{proof}

We will need the following definition.

\begin{definition}[Left contraction]\label{def:leftcont}
For $X\subset \R$ and $y\in \R$ let $X-y=\{x-y\colon x\in X\}$ and let $I=(a,b)$ be an open interval. The \emph{left $I$ contraction} of $X$, denoted $^IX$, is the set $(X \cap (-\infty,a]) \cup ((X \cap (b,\infty)) - (b-a))$. 
\end{definition}

For the next result it's slightly more convenient to work with a Ppp on all of $\R_+$; to distinguish it from $\pi_\lbd$, the stationary distribution for the dynamic Ppp, we'll denote it $\tilde \pi_\lbd$.

\begin{lemma}\label{lem:piAw}
Let $\pi,\tilde \pi_\lbd$ be the measure of a Poisson point process (Ppp) $X$ on $[0,1]$, $\R_+$ respectively with intensity $\lambda$. Let $\lambda\to\infty$ and let $w_\lbd$ be such that $\liminf_\lbd\lambda w_\lbd / \log\lbd>1$ and $\limsup_\lbd w_\lbd < 1$. With $A(w_\lbd)$ as in \eqref{eq:Aw},
\begin{align}\label{eq:piAw}
\pi(A(w_\lbd))=\tilde \pi_\lbd(A(w_\lbd)) \sim \lambda(1-w_\lbd)e^{-\lambda w_\lbd}.
\end{align}
Let $v_\ell=\inf\{x\colon X\cap(x,x+w_\lbd)=\emptyset\}$ and $v_r=\inf (X \cap (v_\ell,\infty))$. Then with respect to $\tilde \pi_\lbd$, 
\enumar
\item conditioned on $v_\ell$, $v_r-v_\ell-w_\lbd$ has distribution exponential$(\lbd)$ and
\item conditioned on $A(w_\lbd)$, with $I:=(v_\ell,v_r)$, the left $I$ contraction of $X$ \\ stochastically dominates a Ppp on $\R_+$ with intensity $\lbd$.
\enumend
\end{lemma}

\begin{lemma}\label{lem:numudens}
Let $\nu_\lbd$ be as in Lemma \ref{lem:renprop} and $\mu_\lbd=\mu_{A(w_\lbd)}$ be as given in  \eqref{eq:muAdef}, i.e., with $q_\lbd(X,dY)$ the transition rate kernel of the dynamic Ppp with intensity $\lbd$ on $[0,1]$, for $B\subset S$,
\[\mu_\lbd(B) := \frac{\int_{X \in A(w_\lbd)} \pi_\lbd(dX) q_\lbd(X,B)}{q_{\exit,\lbd}(A(w_\lbd))},\]
where $q_{\exit,\lbd}(A(w_\lbd)):= \int_{X \in A(w_\lbd)} \pi_\lbd(dX)q_\lbd(X,A(w_\lbd)^c)$. For each $\lbd$, both $\nu_\lbd$ and $\mu_\lbd$ are absolutely continuous with respect to $\pi_\lbd$. Moreover, letting $S_n=\{X \in S\colon |X|=n\}$, the density functions satisfy, $\pi_\lbd$-almost surely,
\[\frac{d\nu_\lbd}{d\pi_\lbd} \le \lf\lbd\rf!e^{\lf\lbd\rf}/ \lbd^{\lf\lbd\rf}, \ \ \frac{d\mu_\lbd}{d\pi_\lbd}\bigg|_{S_n} \le \frac{\lbd+n}{q_{\exit,\lbd}(A(w_\lbd))}.\]
\end{lemma}

\begin{proof}
A Poisson point process on $[0,1]$ conditioned to have cardinality $n$ is equal in distribution to $\{u_i\}_{i=1}^n$ where $(u_i)$ are i.i.d.~uniform on $[0,1]$. Since $\nu_\lbd$ has this distribution with $n=\lf\lbd\rf$, $\nu_\lbd$ has density with respect to $\pi_\lbd$ given by
\[\frac{d\nu_\lbd(X)}{d\pi_\lbd(X)} = \begin{cases} 
\lf\lbd\rf!e^{\lf\lbd\rf}/ \lbd^{\lf\lbd\rf} & \text{if} \ |X|=\lf\lbd\rf, \\
0 & \text{otherwise},
\end{cases}\]
and in particular $\nu_\lbd \ll \pi_\lbd$. For $\mu_\lbd$, since $\pi_\lbd$ is stationary, for $B\subset S$
\[\int_{X\in S}\pi(dX)q(x,B) = \int_{X\in B} \pi(dX)q(X,S).\]
Let $S_n=\{X\in S\colon |X|=n\}$. If $X\in S_n$ then $q(X,S)=\lbd+n$, so for $B\subset S$ and each $n$,
\begin{align*}
\mu_\lbd(B\cap S_n) = \frac{\int_{X \in A(w_\lbd)} \pi_\lbd(dX) q_\lbd(X,B \cap S_n)}{q_{\exit,\lbd}(A(w_\lbd))}&\le q_{\exit,\lbd}(A(w_\lbd))^{-1} \int_{X \in S}\pi_\lbd(dx)q_\lbd(X,B \cap S_n)\\
&= q_{\exit,\lbd}(A(w_\lbd))^{-1} \int_{X\in B\cap S_n} \pi_\lbd(dX)q(X,S) \\
&= \frac{(\lbd+n)\pi_\lbd(B \cap S_n)}{q_{\exit,\lbd}(A(w_\lbd))}.
\end{align*}
This implies the given estimate on the density function $d\mu_\lbd/d\pi_\lbd$ restricted to $S_n$. Since $B=\bigcup_{n\ge 0}B\cap S_n$, if $\pi_\lbd(B)=0$ then $\pi_\lbd(B\cap S_n)=0$ for all $n$, and the above implies $\mu_\lbd(B\cap S_n)=0$ for all $n$, thus $\mu_\lbd(B)=0$; in other words, $\mu_\lbd \ll \pi_\lbd$.
\end{proof}

Proposition \ref{prop:exp} is stated for a sequence $X_m$; we'll generally preserve the notation $X_\lbd$, with the understanding that we are taking an arbitrary sequence $(\lbd_m)$ with $\lbd_m\to\infty$ as $m\to\infty$.
 
\begin{lemma}\label{lem:prop1cond}
If $\liminf_\lbd w_\lbd/\log \lbd$ is large enough then for any sequence $(\lbd_m)$ with $\lbd_m\to\infty$ as $m\to\infty$, conditions (i)-(iii) of Proposition \ref{prop:exp} hold for the dynamic Poisson point process $(X_\lbd)$ with regeneration time $\eta_\lbd$ and distribution $\nu_\lbd$ given by Definition \ref{def:renew} and Lemma \ref{lem:renprop} (a) respectively, and with the gap event $A(w_\lbd)$ defined in \eqref{eq:Aw}.    
\end{lemma}

\begin{proof}
\enumrom
\item letting $p_\lbd=P_{\nu_\lbd}(\tau_{A(w_\lbd)} < \eta_\lbd)$ we will show that $p_\lbd\to 0$ as $\lbd\to\infty$. From any distribution, $\eta_\lbd>0$ almost surely, since if $\tau_{\flsh,\lbd}=0$ then $\tau_{\pin,\lbd}$ is at least the time for $\lf\lbd\rf$ new particles to appear. By Lemma \ref{lem:excprinc},
\[\pi(A(w_\lbd)) = \frac{E_{\nu_\lbd}\left [T_{A(w_\lbd)}(\eta_\lbd)\right]}{E_{\nu_\lbd}[\eta_\lbd]}.\]
Conditioned on $\tau(A(w_\lbd))<\eta_\lbd$, the expected time spent in $A(w_\lbd)$ before $\eta_\lbd$ is at least $1/(2+2\lbd)$, since the number of particles needs to go from $\lf \lbd\rf \pm 1$ to $\lf\lbd\rf$ to refresh and the rate of either transition is at most $2(1+\lf \lbd\rf)$. In other words, $E_{\nu_\lbd}\left [T_{A(w_\lbd)}(\eta_\lbd)\right] \ge  p_\lbd /(2+2\lbd)$. It follows that
\[p_\lbd \le (2+2\lbd) \pi(A(w_\lbd))/E_{\nu_\lbd}[\eta_\lbd].\]
By Lemma \ref{lem:renest} (c), $E_{\nu_\lbd}[\eta_\lbd]=\log \lbd + O(1)>1$ for large $\lbd$ and by Lemma \ref{lem:piAw}, $\pi(A(w_\lbd)) = O( \lbd e^{-\lbd w_\lbd})$ so $p_\lbd = O(\lbd^2 e^{-\lbd w_\lbd}$). If $\liminf_\lbd \lbd w_\lbd / \log \lbd > 2$ then $p_\lbd \to 0$ as $\lbd\to\infty$.\\
\item With stopping sequence $(t_{\lbd,n}^\ren)$ and waiting times $w_{\lbd,n}^\ren = t_{\lbd,n}^\ren - t_{\lbd,n-1}^\ren$ as in Lemma \ref{lem:renprop} and excess time $w_{\lbd,n}^*:=(w_{\lbd,n}^\ren-\log\lf\lbd\rf)_+$, by Lemma \ref{lem:renest} (d) with $N_\lbd(0)=\lf\lbd\rf$, for any $\theta<c$, $E_{\nu_\lbd}[e^{\theta w_{\lbd,1}^*}]=:e^{C_1} < \infty$. Since, by Lemma \ref{lem:renprop} (c), $(w_{\lbd,n}^\ren)$ are i.i.d., so are $(w_{\lbd,n}^*)$, so 
\[E_{\nu_\lbd}\left [\exp\left(\theta \sum_{i=1}^n w_{\lbd,i}^*\right)\right] = e^{nC_1},\]
and by Markov's inequality applied to $e^{\theta \cdot}$,
\[P_{\nu_\lbd}\left(\sum_{i=1}^n w_{\lbd,i}^* > t\right) \le e^{-\theta t + nC}.\]
Since $E_{\nu_\lbd}[ t_{\lbd,n}^\ren] = n E_{\nu_\lbd}[\eta_\lbd] = n(\log \lf\lbd\rf + O(1))$ by Lemma \ref{lem:renest} (c), which $=n(\log\lf\lbd\rf+O(1))$, for some $C_2>0,$
\begin{align}\label{eq:trentail}
P_{\nu_\lbd}(t_{\lbd,n}^\ren / E_{\nu_\lbd}[t_{\lbd,n}^\ren] > t) &\le P_{\nu_\lbd}(t_{\lbd,n}^\ren > t (\log\lf\lbd\rf-C_2)n)  \nonumber \\
&\le P_{\nu_\lbd}\left(\sum_{i=1}^{n} w_{\lbd,i}^*  > ((t-1)\log\lf\lbd\rf - C_2)n  \right) \nonumber \\
&\le \exp((-\theta (t-1)\log \lf\lbd\rf + C_1+\theta C_2)n) \nonumber \\
&\le \exp(-(\theta (t-1)\log \lf\lbd\rf )n/2)
\end{align}
for any $n\ge 1$ if $t\ge 2$ and $\theta \log \lf\lbd\rf/2 \ge C_1+\theta C_2$. This holds for large $\lbd$ and implies uniform integrability of $\{t_{\lbd,n}^\ren/E_{\nu_\lbd}[t_{\lbd,n}^\ren]\colon n\ge 1, \ \lbd\ge\lbd_0\}$ for some $\lbd_0$, which implies condition (ii) for any sequence $(\lbd_m)$ with $\lbd_m\to\infty$ as $m\to\infty$.\\

\item With $N_\lbd=\inf\{n\colon \tau_{A(w_\lbd)} < t_{\lbd,n}^\ren\}$, $n_\lbd(\dlt)=\min\{n\colon P(N_\lbd< n)\ge \dlt\}$, using the definition of conditional probability and \eqref{eq:trentail}, for some $\theta,\lbd_0$ and all $t\ge 2$ and $\lbd\ge\lbd_0$,
\begin{align}\label{eq:trenctail}
&P_{\nu_\lbd}(t_{\lbd,n_\lbd(\dlt)}^\ren / E_{\nu_\lbd}[t_{\lbd,n_\lbd(\dlt)}^\ren] > t \mid N_\lbd < n_\lbd(\dlt)) ) \\
&\le \exp(-(\theta (t-1)\log \lf\lbd\rf )n_\lbd(\dlt)/2 + \log (1/\dlt)).
\end{align}
By a union bound over regeneration cycles, $p_\lbd (n_\lbd(\dlt)-1)\ge \dlt$ so $n_\lbd(\dlt)\ge \dlt/p_\lbd = (\dlt^2/p_\lbd) (1/\dlt)$ and since $p_\lbd\to 0$, taking $\dlt_\lbd=p_\lbd^{1/2}$, $n_\lbd\ge 1/\dlt$ and $n_\lbd\to\infty$ as $\lbd\to\infty$. With $c_\lbd=\theta \log\lf\lbd\rf/2$, for large $\lbd$, $c_\lbd\ge 1$. Since $\log(x)\le x$ for all $x$, $\log(1/\dlt_\lbd)\le n_\lbd(\dlt_\lbd)$ so for large $\lbd$ the probability in \eqref{eq:trenctail} is at most $\exp(-c_\lbd(t-2) n_\lbd(\dlt_\lbd))$.  Using the trivial bound of $1$ for the above probability for $t<3$ and the above bound for $t\ge 3$ and integrating over $t$ we find that
\[E_{\nu_\lbd}\left[t_{\lbd,n_\lbd(\dlt_\lbd)}^\ren / E_{\nu_\lbd}[t_{\lbd,n_\lbd(\dlt_\lbd)}^\ren] \mid N_\lbd < n_\lbd(\dlt_\lbd)) \right]= O(1),\]
which gives condition (iii).
\enumend
\end{proof}

\begin{lemma}\label{lem:exitest}
Let $A(w_\lbd)$ be as in \eqref{eq:Aw} and $\pi_\lbd$ the measure of a Ppp on $[0,1]$. In the notation of Proposition \ref{prop:mustat}, $(\tau_n(A(w_\lbd),A(w_\lbd)^c)/n)_n$ is uniformly integrable for any distribution of $X_\lbd(0)$ and if $\lbd w_\lbd\to\infty$ then the exit rate $q_\exit(A(w_\lbd))$ satisfies
\[q_\exit(A(w_\lbd)) \sim \lbd w_\lbd \pi_\lbd (A(w_\lbd)).\]
In particular, $q_\exit(A(w_\lbd))>0$ for large $\lbd$ and with $\mu_\lbd$ as in Lemma \ref{lem:numudens}, \[E_{\mu_\lbd}[\tau_\exit(A(w_\lbd))] = 1/q_\exit(A(w_\lbd)).\]
\end{lemma}

\begin{proof}
The last implication follows from Proposition \ref{prop:mustat} (note that the required ergodicity is proved in Lemma \ref{lem:dynerg}). For tidiness denote $\tau_n(A(w_\lbd),A(w_\lbd)^c)$ by $\tau_n$. To show that $(\tau_n/n)$ is UI it suffices that for some sequence $(\sigma_n)_{n\ge 0}$ for which $\tau_n\le\sigma_n$ a.s.~for every $n\ge 1$, $\sigma_0<\infty$ a.s.~and the waiting times $w_n=\sigma_n-\sigma_{n-1}$, have, for some $\theta,C>0$ and all $n\ge 1$, $E[e^{\theta w_n} \mid w_i \ \forall i<n]\le C<\infty$, since then $E[e^{\theta (\sigma_n-\sigma_0)}] \le C^n$ and
\[P(\sigma_n/n > a) \le P(\sigma_0>na/2)+ P(e^{\theta \sigma_n} > e^{na/2}) \le P(\sigma_0>na/2) + e^{-na/2 + n\log C},\]
which $\to 0$ uniformly in $n$ as $a\to\infty$. Such a sequence is given as follows. Let $\sigma_0=\eta_\lbd$, which by Lemma \ref{lem:renest} (d) is a.s.~finite, for any initial distribution. 
Given $\sigma_n$, define the sequence $(t_{n,k})_{k\ge 0}$ by $t_{n,0}=\sigma_n$ and $t_{n,k+1} = \inf\{t>t_{n,k}+1\colon N_\lbd(t)=\lf\lbd\rf\}$. 
Let
\[K_n=\inf\{k\colon N_\lbd(t_{n,k-1}+1)=0 \ \text{and} \ X_\lbd(t_{n,k}) \notin A(w_\lbd)\}\]
and let $\sigma_{n+1}=t_{n,K}$. Then $X_\lbd(t_{n,K-1}+1) \in A(w_\lbd)$ and $X_\lbd(t_{n,K}) \notin A(w_\lbd)$ so for each $n\ge 0$ there is a transition from $A(w_\lbd)$ to $A(w_\lbd)^c$ in the time interval $(\sigma_n,\sigma_{n+1}]$, which implies $\tau_n\le \sigma_n$ for all $n\ge 1$. It remains to show that $E[e^{\theta w_n} \mid w_i \ \forall i<n] \le C$ for some $\theta,C>0$ and all $n$. Since $N_\lbd(\sigma_n)=\lf\lbd\rf$ for each $n$, by the strong Markov property it suffices to show that if $N_\lbd(0)$ a.s.~then with $s_0=0$, $s_{k+1}=\inf\{t>s_k+1\colon N_\lbd(t)=\lf\lbd\rf\}$ and $\kappa$ defined in the same way as $K_n$ except with $(s_k)_k$ in place of $(t_{n,k})_k$, for some $\theta>0$, $E[\exp(\theta s_\kappa)]$ is deterministic and finite. For this to hold it suffices that
\enumrom
\item $\kappa$ is geometrically distributed with some parameter $p>0$ and that
\item with $B=\{\kappa=1\}$, if $N_\lbd(0)=\lf\lbd\rf$ a.s.~then for some $\theta_0,C$, \\ $E[e^{\theta s_1} \mid B] \ \vee \ E[e^{\theta s_1} \mid B^c]\le C$ a.s.
\enumend
To see that this suffices, first note that by dominated convergence there is $0<\theta\le \theta_0$ such that $C_\theta:=E[e^{\theta s_1} \mid B] \ \vee \ E[e^{\theta s_1} \mid B^c]<1/(1-p)$. Then observe that since $N_\lbd(s_k)=\lf\lbd\rf$ a.s.~for every $k$, with $\F$ the natural filtration of $X_\lbd$, (ii) implies that almost surely for each $k$
\[E[\exp(\theta (s_k-s_{k-1}) \mid \F(s_{k-1}), \kappa>k] \ \vee \  E[\exp(\theta (s_k-s_{k-1}) \mid \F(s_{k-1}), \kappa=k]\le C_\theta,\]
so with
\[M_j = E[e^{\theta (s_j-s_{j-1})} \mid \kappa>j, s_1,\dots,s_{j-1}] \ \ \text{for} \ j<k \ \ \text{and} \ \ M_kE[e^{\theta (s_j-s_{j-1})} \mid \kappa=k, s_1,\dots,s_{k-1}],\]
it follows that $M_j\le C_\theta$ for all $j\le k$ and
\[E[e^{\theta s_k}\mid \kappa=k] = E\left[\prod_{j=1}^k M_j\right] \le E[C_\theta^k] = C_\theta^k.\]
Then combining with (i) and using $C_\theta<1/(1-p)$,
\[E[e^{\theta s_\kappa}] = \sum_k E[e^{\theta s_k}\mid \kappa=k] P(\kappa=k) \le \sum_k C_\theta^k(1-p)^{k-1}p < \infty.\] 
It remains to show (i) and (ii). Recalling that $B=\{\kappa=1\}$, (i) holds with $p=P(B\mid N_\lbd(0)=\lf\lbd\rf)$ provided the latter is positive and deterministic, since with $\F$ the natural filtration of $X_\lbd$, by the Markov property and the fact that $N_\lbd(s_{k-1})=\lf\lbd\rf$ a.s.,
\[P(\kappa=k \mid \kappa>k-1, \ \F(s_{k-1}))=P(\kappa=k \mid X_\lbd(s_{k-1}), \kappa>k-1)=p,\]
and so
\[P(\kappa=k \mid \kappa>k-1) = E[P(\kappa=k \mid \kappa>k-1, \ \F(s_{k-1})] = E[p]=p.\]
To see that $p$ is positive and deterministic note that, since $N_\lbd$ is an irreducible Markov chain, $P(N_\lbd(1)=0\mid N_\lbd(0)=\lf\lbd\rf)$ is positive and deterministic, and that $P(X_\lbd(s_1) \notin A(w_\lbd) \mid X_\lbd(1) =\emptyset)$ is deterministic, and is positive so long as $\lf\lbd\rf > \lceil 2/w_\lbd\rceil$ (which holds for large $\lbd$ by the assumption $\lbd w_\lbd\to\infty$), since then $X_\lbd(s_1) \notin A(w_\lbd)$ if there is one point in each interval $[(i-1)/\lceil 2/w_\lbd, i/\lceil 2/w_\lbd]$, $i=1,\dots,\lceil 2/w_\lbd\rceil$. Using the Markov property at time $1$ and combining the two parts then implies the claim.\\

For (ii), let $B_1=\{N_\lbd(1)=0\}$ and $B_2=\{X_\lbd(s_1) \notin A(w_\lbd)\}$ so that $B=B_1\cap B_2$, and note that by Lemma \ref{lem:countest} (a),
\[P(s_1-1 > t + \log(\lbd)/2 \mid B_1) \le Ce^{-ct},\]
and in particular $E[e^{\theta s_1} \mid B_1]<\infty$ for $\theta<1$. Since $y(t):=E[N_\lbd(t)]$ satisfies the equation $dy/dt = \lbd - y$, which has solution $y(t)=\lbd + e^{-t}(y(0)-\lbd)$, if $N_\lbd(0)=\lf\lbd\rf$ then $E[N_\lbd(t)]\le\lbd$ for $t>0$ and with the trivial bound $|a-b|\le |a|+|b|$, $E[\lbd^{-1/2}|N_\lbd(t)-\lbd|] \le 2\sqrt{\lbd}$ and with $q=P(B_1)>0$,
\[E[\lbd^{-1/2}|N_\lbd(1)-\lbd| \mid N_\lbd(0)=\lf\lbd\rf, \ N_\lbd(1)\ne 0] \le \frac{2\sqrt{\lbd}}{1-q},\]
and using again Lemma \ref{lem:countest} (a)
\[P(s_1-1 > t + \log(2\sqrt{\lbd}/(1-q)) \mid B_1^c) \le Ce^{-ct},\]
and in particular $E[e^{\theta s_1} \mid B_1^c]<\infty$ for $\theta<1$. Since $B_1^c \subset B^c$ it remains to show that $E[e^{\theta s_1} \mid B_1 \cap B_2]<\infty$ and $E[e^{\theta s_1} \mid B_1 \cap B_2^c]<\infty$ for some $\theta>0$. Since we have shown that $E[e^{\theta s_1} \mid B_1]<\infty$ for $\theta<1$ it suffices to show that, conditioned on $B_1$, $B_2$ is independent of $s_1$, for which it is enough to show that if $X_\lbd(0)=\emptyset$ then $X_\lbd(\eta_\lbd)$ is independent of $\eta_\lbd$. In the notation of the proof of Lemma \ref{lem:renprop}, if $X_\lbd(0)=\emptyset$ then $\eta_\lbd=\inf\{t\colon |K_\lbd(t)|=\lf\lbd\rf\}$ which is determined by $(w_i),(s_i)$, and so $K_\lbd(\eta_\lbd)$ is also determined by $(w_i),(s_i)$. Conditioned on $\eta_\lbd$ and $K_\lbd(\eta_\lbd)$, $X_\lbd(\eta_\lbd)=\{u_i\}_{i \in K_\lbd(\eta_\lbd)}$ has the distribution of $|K_\lbd(\eta_\lbd)|=\lf\lbd\rf$ i.i.d.~uniform points on $[0,1]$, so the claim holds.\\

We now compute the estimate for the exit rate. By definition
\begin{align*}
q_\exit(A(w_\lbd)) &= \int_{X \in A(w_\lbd)}\pi_\lbd(dx)q_\lbd(X,A(w_\lbd)^c) \\
&= \pi(A(w_\lbd)) E_{\pi_\lbd}[ q_\lbd(X,A(w_\lbd)^c) \mid A(w_\lbd)].
\end{align*}
Let $\tilde\pi_\lbd$ be a Ppp$(\lbd)$ on $\R_+$, recalling that, in the notation of Lemma \ref{lem:piAw}, $A(w_\lbd)=\{v_\ell \le 1-w_\lbd\}$ and that $\tilde \pi_\lbd(A(w_\lbd))=\pi(A(w_\lbd))$. For $X\in A(w_\lbd)$,
\begin{align*}
\lbd w_\lbd \ge q_\lbd(X,A(w_\lbd)^c) &= \lbd (w_\lbd - (v_r\wedge 1 -v_\ell-w_\lbd))_+ \\
&\ge  \lbd (w_\lbd - (v_r-v_\ell-w_\lbd))_+.
\end{align*}
Let $y_\lbd=v_r-v_\ell-w_\lbd$. Conditioned on $v_\ell \le 1-w_\lbd$, by Lemma \ref{lem:piAw} $y_\lbd$ is exponentially distributed with rate $\lbd$. This implies that
\[\pi_\lbd(|q_\lbd(X,A^c)- \lbd w_\lbd|\ge \lbd s \mid A(w_\lbd))\le e^{-\lbd s},\]
and thus that
\[ \left| E_{\pi_\lbd}[q_\lbd(X,A(w_\lbd)^c) \mid A(w_\lbd)] - \lbd w_\lbd \right| \le \int_0^\infty e^{-s}ds=1 = o(\lbd w_\lbd),\]
completing the proof.
\end{proof}

\begin{lemma}\label{lem:denstime}
Suppose for some $\alpha>0$ and every $\lbd$ that $X_\lbd(0)$ is $\lbd\alpha$-dense at scale $w_\lbd/2$. If $\liminf_\lbd \lbd w_\lbd/\log\lbd $ is large enough then for some $c>0$,
\[P(\tau_{A(w_\lambda)} \le e^{c\alpha\lambda w_\lambda/2}) \to 0 \ \ \text{as} \ \lbd\to\infty.\]
\end{lemma}

\begin{proof}
For $i=1,\dots, n_\lambda:=\lf 2/w_\lambda \rf$ let $N_i(t)=|X(t)\cap [(i-1)w_\lambda/2,\, i w_\lambda/2]|$ and let $N_{n_\lambda+1}(t)=|X(t)\cap [1-w_\lambda/2,1]$. If $X(0)$ is $\alpha$-dense at scale $w_\lambda/2$ then $N_i(0) \ge \alpha \lambda w_\lambda/2$ for all $i \le n_\lambda+1$, and if $N_i(t)>0$ for all $i\le n_\lambda+1$ then $X(t)$ has no gap of size $w_\lambda$. Using Lemma \ref{lem:countest} (d) and taking a union bound over $i$, for some $c,C>0$,
\begin{align*}
P(\tau_{A(w_\lambda)} \le e^{c\alpha\lambda w_\lambda/2}) &\le P(N_i(t)=0 \ \text{for some} \ i\le n_\lambda+1, \ t \le e^{c \alpha\lambda w_\lambda/2}) \\
&\le (n_\lambda+1)Ce^{-c \alpha \lambda w_\lambda/2}.
\end{align*}
Since $n_\lbd = O(1/w_\lbd)=O(\lbd)$, if $\lambda w_\lambda / \log(\lambda) > 2/(c\alpha)$ the above bound is $O(\lbd^{-\dlt})$ for some $\delta>0$ and thus $\to 0$ as $\lambda\to\infty$.
\end{proof}

\begin{lemma}\label{lem:prop3cond}
If $\liminf_\lbd \lbd w_\lbd/\log \lbd$ is large enough and $\limsup_\lbd w_\lbd<1$ then conditions (i)-(iii) of Proposition \ref{prop:tauex} hold for the dynamic Poisson point process.
\end{lemma}

\begin{proof}
\textit{Condition (i). } Let $\mu_\lbd$ be as given in  Lemma \ref{lem:numudens}. We wish to show that $P_{\mu_\lbd}(\eta_\lbd < \tau_{A(w_\lbd)})\to  1$ as $\lbd\to\infty$, for which it suffices that
\enumalph
\item for large $C>0$, $P_{\mu_\lbd}(\eta_\lbd < C \log \lbd) \to 1$ as $\lbd\to\infty$ and that
\item for some $c>0$, $P_{\mu_\lbd}(\tau_{A(w_\lbd)} > e^{c \lbd w_\lbd/4}) \to 1$.
\enumend

For (a) we will use the bound on $d\mu_\lbd/d\pi_\lbd$ from Lemma \ref{lem:numudens} together with Lemma \ref{lem:renest} (d). For $C_1>0$ to be determined, first decompose and bound as
\begin{align}
P_{\mu_\lbd}(\eta_\lbd > 2C_1\log\lbd) \le P_{\mu_\lbd}(\eta_\lbd > 2C_1\log \lbd \mid N_\lbd(0)\le C_1\lbd) + \sum_{n=\lf C_1 \lbd\rf+1}^{\infty} P_{\mu_\lbd}(N_\lbd(0)=n). 
\end{align}
We need to show each of the two parts on the right-hand side tends to $0$ as $\lbd\to\infty$.\\

\noindent\textit{First part.} If $n\le C_1\lbd$ and $C_1>1$ then for large $\lbd$, $2C_1\log\lbd > \log(n \vee \sqrt{\lbd}) + C_1\log\lbd$ so $P_{\nu_\lbd}(\eta_\lbd > 2C_1\log\lbd \mid N_\lbd(0)\le C_1\lbd) \le P_{\mu_\lbd}(\eta_\lbd > \log(N_\lbd(0)\vee \sqrt{\lbd}) + C_1\log\lbd \mid N_\lbd(0)\le C_1\lbd)$, which by Lemma \ref{lem:renprop} (d) is $\le Ce^{-cC_1\log\lbd} = C\lbd^{-cC_1}$ for some $c,C>0$.\\

\noindent\textit{Second part. }By the estimates on $q_\exit(A(w_\lbd))$ and $\pi_\lbd(A(w_\lbd))$ in Lemma \ref{lem:exitest}, \ref{lem:piAw} respectively and since $\lbd w_\lbd\to\infty$ and $\limsup_\lbd w_\lbd <1$ by assumption, for large $\lbd$, $q_{\exit,\lbd}(A(w_\lbd)) \ge e^{-\lbd w_\lbd}$. Combining with the bound on $d\mu_\lbd/d\pi_\lbd$ from Lemma \ref{lem:numudens}, for large $\lbd$
\[P_{\mu_\lbd}(N_\lbd(0)=n) \le (\lbd+n)e^{\lbd w_\lbd}P_{\pi_\lbd}(N_\lbd(0)=n) = (\lbd+n)e^{\lbd w_\lbd} \frac{\lbd^n}{n!} e^{-\lbd}.\]
Bounding by an integral, $\log n! \ge (n-1)(\log(n-1)-1)$, so for large $\lbd$
\begin{align*}
\log P_{\mu_\lbd}(N_\lbd(0)=n) &\le \log(\lbd+n) - \lbd (1-w_\lbd) + n\log\lbd - (n-1)(\log (n-1) - 1)\\
&\le 2\log(\lbd+n) - (n-1)(\log(n-1)-(1+\log\lbd)).
\end{align*}
If $n\ge C_1\lbd$ for large $C_1$ then
\[\log P_{\mu_\lbd}(N_\lbd(0)=n) \le 2\log(2n) - (n-1)\log(C_1)/2 \le -n\log(C_1)/3\]
for large $n$, so for large $\lbd$
\[\sum_{n=\lf C_1\lbd\rf+1}^\infty P_{\mu_\lbd}(N_\lbd(0)=n) \le \sum_{n=\lf C_1\lbd\rf+1}^\infty(C_1)^{-n/3} = O\left( C_1^{-C_1\lbd/3}\right),\]
which tends to $ 0$ as $\lbd\to\infty$, completing the proof of (a).\\

For (b), we first show how to obtain a $\mu_\lbd$-distributed set by biasing and then adding a point to a $\pi_\lbd$-distributed set, and also upperbound the bias. Use $x,y,z$ for points in $\R_+$ and $X,Y,Z$ for discrete subsets of $\R_+$. Then
\begin{align*}
& \mu_\lbd(B) = \int_{Y \in S}\xi_\lbd(dY)\frac{q_\lbd(Y,B\cap A(w_\lbd)^c)}{q_\lbd(Y,A(w_\lbd)^c)} \quad \text{where} \quad  \\
& \xi_\lbd(dY) = \1(Y \in A(w_\lbd))\pi_\lbd(dY)\frac{q_\lbd(Y,A(w_\lbd)^c)}{q_{\exit,\lbd}}.
\end{align*}
With $v_\ell=\inf\{x\colon Y\cap(x,x+w_\lbd)=\emptyset\}$ and $v_r=1\wedge \inf (Y \cap (v_\ell,\infty))$, if $Y\in A(w_\lbd)$, equivalently, $v_\ell \le 1-w$, then the outgoing transition rate satisfies
\begin{align}\label{eq:qlbdAc}
q_\lbd(Y,A(w_\lbd)^c)=\lbd (w_\lbd-(v_r-v_\ell-w_\lbd))_+ = \lbd(2w_\lbd-(v_r-v_\ell))_+,
\end{align}
so using the estimate of $q_{\exit,\lbd}$ from Lemma \ref{lem:exitest} and noting the conditioning on $A(w_\lbd)$ below, uniformly over $Y\in A(w_\lbd)$ as $\lbd\to\infty$,
\begin{align*}
\frac{d\xi_\lbd(Y)}{d\pi_\lbd(Y \mid A(w_\lbd))} \bigg|_{A(w_\lbd)} \sim 2 -\frac{v_r-v_\ell}{w_\lbd},
\end{align*}
and in particular, for large $\lbd$ the density function of $\xi_\lbd$ with respect to $\pi_\lbd(\cdot \mid A(w_\lbd))$ is bounded above by $3$, and so for large $\lbd$
\begin{align}\label{eq:xipidens}
\xi_\lbd(B) \le 3 \pi_\lbd(B\mid A(w_\lbd)) \ \ \text{for measurable} \ \ B.
\end{align}
For $Y\in A(w_\lbd)$ satisfying $v_r-v_\ell<2w$, the transition probability measure taking $\xi_\lbd$ to $\mu_\lbd$,
\[\frac{q_\lbd(Y,\cdot \cap A(w_\lbd)^c)}{q_\lbd(Y,A(w_\lbd)^c)}\]
has distribution $Y\cup\{v_m\}$ where $v_m$ is uniform on $[v_r-w_\lbd,v_\ell+w_\lbd]$; if $v_r-v_\ell\ge 2w$ then $q_\lbd(Y,A(w_\lbd)^c)=0$.\\

For what follows it will be convenient to extend point sets to $\R_+$; note that if $\tilde \pi_\lbd$ is a Ppp with intensity $\lbd$ on $\R_+$ and $\tilde \xi_\lbd, \tilde \mu_\lbd$ are defined from $\tilde \pi_\lbd$ using the formulas above, with $q_\lbd(Y,A(w_\lbd)^c)$ defined from \eqref{eq:qlbdAc} and the definition of $A(w_\lbd)$ unchanged, then $\tilde \xi_\lbd(\{X\colon X\cap [0,1] \in B) = \xi_\lbd(B)$ for measurable $B$ and similarly for $\mu_\lbd$. The dynamic Ppp can also be extended to $\R_+$ by running an independent copy on each interval $[n,n+1]$, $n\in \N$, and its restriction to $[0,1]$ coincides with the dynamic Ppp on $[0,1]$.\\

So, let $Y$ have distribution $\tilde \xi_\lbd$ with $v_\ell,v_r$ as above and conditioned on $Y$ let $v_m$ be uniform on $[v_r-w_\lbd,v_\ell+w_\lbd]$, so that $Y\cup\{v_m\}$ has distribution $\tilde \mu_\lbd$ and let $(X_\lbd(t))_{t \in \R_+}$ be the dynamic Ppp on $\R_+$ with intensity $\lbd$ and initial value $X_\lbd(0)=Y\cup \{v_m\}$, with $P$ denoting the measure of $X_\lbd$.
With $I=(v_\ell,v_r)$, let $\tau_1 = \inf\{t\colon \,^I X(t) \in A(w_\lbd/2)\}$ and with $\Delta=v_r-v_\ell$, for $j=1,\dots,4$ let $M_j(t) = |X_\lbd(t) \cap (v_\ell+(j-1)\Delta, v_\ell + j \Delta)|$, let $\sigma_j=\inf\{t\colon M_j(t)>0$ and let $\tau_{2,j}=\inf\{t>\sigma_j\colon M_j(t)=0\}$. Let $\sigma=\min_j \sigma_j$, $\tau_2=\min_j \tau_{2,j}$ and $\zeta = \inf\{t\colon v_m \notin X_\lbd(t)\}$. If $v_r-v_\ell<2w_\lbd$, which holds $\tilde\mu_\lbd$ almost surely, then $\Delta \le w_\lbd/2$. If in addition $\sigma \le \zeta$ then $\tau_{A(w_\lbd)} \ge \tau_1 \wedge \tau_2$. So, (b) will be proved if we can show that for some $c>0$, as $\lbd\to\infty$
\enumar
\item $P(\tau_1 \le e^{c \lbd w_\lbd/4}) \to 0$,
\item $P(\tau_2 \le e^{c \lbd w_\lbd/4})\to 0$ and
\item $P(\zeta<\sigma)\to 0$.
\enumend
We prove these in order.
\enumar
\item By statement 2 in Lemma \ref{lem:piAw}, if $Y$ has distribution $\tilde \pi(\cdot \mid A(w_\lbd)$ and $I$ is defined from $Y$ as above then $^IY$ dominates a Ppp on $\R_+$ with intensity $\lbd$, and if $B$ denotes the event of being $\lbd\alpha$-dense at scale $w_\lbd/2$ on the interval $[0,1]$ then by Lemma \ref{lem:alphadense},
\[\tilde\pi_\lbd( ^IY \in B^c) \to 0 \ \text{as} \ \lbd\to\infty.\]
Using \eqref{eq:xipidens} the same is true with $\tilde \xi_\lbd$ in place of $\tilde \pi_\lbd$. Combining with Lemma \ref{lem:denstime} (using $w_\lbd/2$ in place of $w_\lbd$), the claim follows.
\item Use Lemma \ref{lem:countest} (b), then (d), on $M_j$, with $\rho=\lbd (v_r-v_\ell)\ge\lbd w_\lbd/4$, to find that $P(\tau_{2,j}\le \sigma_j+e^{c \lbd w_\lbd/4}) \le 4C/(\lbd w_\lbd) + Ce^{-c\lbd w_\lbd/4}=o(1)$, then take a union bound over $j\in \{1,\dots,4\}$.
\item $\zeta, \sigma_j$ is exponentially distributed with respective rate $1, \lbd(v_r-v_\ell)/4 \ge \lbd w_\lbd/4$, so $P(\zeta<\sigma_j) \le 1/(1+\lbd w_\lbd/4)$ and $P(\zeta<\sigma) = P(\zeta<\sigma_j \ \text{for some} \ j) \le 4/(1+\lbd w_\lbd/4) \to 0$ as $\lbd\to\infty$.
\enumend

\nid\textit{Condition (ii). }We want to show that $E_{\mu_\lbd}[\eta_\lbd] = o(E_{\mu_\lbd}[\tau_\exit(A(w_\lbd))])$. From the proof of part (a) of condition (i) above we have the estimate
\[P_{\mu_\lbd}(\eta_\lbd)>2C_1\log \lbd) \le C\lbd^{-cC_1} + CC_1^{-C_1\lbd/3}\]
for some $c,C>0$ and large $\lbd$, which for large $\lbd$ is $O(e^{-C_1})$, from which we deduce that
\[E_{\mu_\lbd}[\eta_\lbd]=O(\log\lbd).\]
On the other hand, by the estimates of $q_\exit(A(w_\lbd))$ and $\pi_\lbd(A(w_\lbd))$ from Lemmas \ref{lem:exitest}, \ref{lem:piAw},
\[E_{\mu_\lbd}[\tau_\exit(A(w_\lbd))] = 1/q_\exit(A(w_\lbd))\sim \frac{e^{\lbd w_\lbd}}{\lbd^2 w_\lbd(1-w_\lbd)}\ge \frac{e^{\lbd w_\lbd / 2}}{\lbd }\]
for large $\lbd$ (using $\lbd w_\lbd\to\infty$ and $e^x/x\ge e^{x/2}$ for large $x$) and if $\liminf_\lbd \lbd w_\lbd / \log\lbd >4$ the above is $\ge \lbd$ for large $\lbd$, completing the verification of condition (ii).\\

\nid\textit{Condition (iii). } From Lemma \ref{lem:piAw}, $\pi_\lbd(A(w_\lbd))\sim \lbd (1-w_\lbd)e^{-\lbd w_\lbd}\to 0$ as $\lbd\to\infty$ if $\lbd w_\lbd / \log\lbd > 1$.
\end{proof}

\begin{proof}[Proof of Proposition \ref{prop:iv}]
We prove statements 1 and 2 in that order.
\enumar
\item Let $C=2\limsup_\lbd \log(N_\lbd(0))/\lbd$ and let $t_\lbd=3C\lbd/4$, then use Lemma \ref{lem:renest} (a) with $t_0=t_\lbd$ and use Lemma \ref{lem:renest} (c) with $t_0=0$ and $t=(2/c)\log \lbd$. By assumption, $N_\lbd(0) e^{-t_\lbd} \to 0$ for each $\ep>0$ and $\lbd e^{-ct}=1/\lbd\to 0$, so $P(\eta_\lbd > t+t_\lbd\to 0$ as $\lbd\to\infty$. Since $t=o(t_\lbd)$, $P(\eta_\lbd > C\lbd)\to 0$ as $\lbd\to\infty$.
\item By the $\alpha$-dense assumption and Lemma \ref{lem:denstime} it is enough to show that for fixed $\ep>0$ and large enough $\liminf_\lbd \lbd w_\lbd/\log\lbd$, $P(\eta_\lbd \le e^{\ep \lbd w_\lbd}) \to 1$ as $\lbd\to\infty$. If $\liminf_\lbd \lbd w_\lbd/\log \lbd> 1/\ep $ then for any $C>0$, $e^{\ep \lbd w_\lbd} > C\lbd$ for large $\lbd$ so the claim follows from statement 1.
\enumend
\end{proof}

\begin{proof}[Proof of Proposition \ref{prop:main}]
Statement (i) of Theorem \ref{thm:main} with initial distribution $\nu_\lbd$ follows directly from Lemma \ref{lem:prop1cond} and Proposition \ref{prop:exp}. For statement (ii), use Lemma \ref{lem:exitest}, \ref{lem:prop3cond} which, together with ergodicity of the dynamic Ppp which is proved in Lemma \ref{lem:dynerg}, verify the conditions of Propositions 2 and 3, to find that $E_{\nu_\lbd}[\tau(A(w_\lbd))] \sim 1/q_\exit(A(w_\lbd))$, then use the estimate of $q_\exit(A(w_\lbd))$ from Lemma \ref{lem:exitest} combined with the estimate of $\pi_\lbd(A(w_\lbd))$ from Lemma \ref{lem:piAw} to conclude.
\end{proof}

\section{Estimates of $\pi_\lbd$}\label{sec:piest}

Here we prove Lemma \ref{lem:piAw}, the estimates on $\pi_\lbd(A(w_\lbd))$ and the conditional distribution $\pi_\lbd(\cdot\mid A(w_\lbd))$, for which we'll need a couple of preliminary results. We begin with the following obvious result, that we state without proof. It is the analogue, for i.i.d.~sequences, of the Poisson thinning property.

\begin{lemma}\label{lem:iidthin}
Let $(y_n)_{n\ge 1}$ be an i.i.d.~sequence of random variables on $\R$ and $B\subset \R$ a Borel set, and let $j_0=k_0=0$ and $j_i=\inf\{n>j_{i-1}\colon y_n \notin B\}$, $k_i=\inf\{n>k_{i-1}\colon y_n\in B\}$. Then $(k_i-k_{i-1})_{i\ge 1}$ is an i.i.d.~geometric$(P(y_1\in B))$ sequence and conditioned on $(k_i)$, $(y_{j_i})_{i\ge 1}$ and $(y_{k_i})_{i\ge 1}$ are independent i.i.d.~sequences, with the distribution of $y_{j_i}$, $y_{k_i}$ equal to the distribution of $y_1$ conditioned on $\{y_1\notin B\}$, $\{y_1\in B\}$, respectively.
\end{lemma}

We also require the following sufficient condition for a point set with given spacing to stochastically dominate a Poisson point process.

\begin{lemma}\label{lem:ratedom}
Let $(w_n)$ be a sequence of $\R_+$-valued random variables, with $t_n:=\sum_{i=1}^n w_i$, and define
\[r_n(t) = -\liminf_{s \to 0^+}s^{-1}\log P(w_n > t+s \mid w_n>t, w_1,\dots,w_{n-1}).\]
If $P(r_n(t) \ge \lbd \ \forall n\ge 1, \ \forall t\in \R_+)=1$ then with $\zeta:=\lim_{n\to\infty} t_n \in [0,\infty]$ the point process $\{t_n\}_{n\ge 1}$ stochastically dominates a Poisson point process on $[0,\zeta)$ with rate $\lbd$, in the sense that there is a coupling of $\{t_n\}$ with a Ppp $X$ on $\R_+$ with rate $\lbd$ such that $\{t_n\} \supset X\cap [0,\zeta)$ with probability $1$.
\end{lemma}

\begin{proof}
Define $G_n(t)=e^{\lbd t}P(w_n>t \mid w_1,\dots,w_{n-1}) \ge 0$. Then
\[\log G_n(t) = \lbd t + \log P(w_n>t\mid w_1,\dots,w_{n-1}),\]
so
\[-\liminf_{s\to 0^+} s^{-1}\log G_n(t+s) = -\lbd + r_n(t)\ge 0\]
for all $n,t$ almost surely, since $r_n(t)\ge\lbd$ by assumption. Since $\log G_n$ has non-positive upper derivative it is non-increasing, so $G_n$ is non-increasing and since $G_n(0)=1$, $G_n(t)\le 1$. Let $X$ be a Poisson point process with intensity $\lbd$ and $(u_i)$ an independent, i.i.d.~uniform$[0,1]$ sequence. Let $s_n= G^{-1}(u_n)$ so that $P(s_n>t)=G(t)$. \\

Define $(w_i'),(t_i')$ as follows: let $t_n'=\sum_{i=1}^n w_i'$ and conditioned on $w_1',\dots,w_{n-1}'$ let $y_n=\inf (X\cap[t_{n-1}',\infty))$ and let $w_n'= s_n \wedge y_n$. Then
\begin{align*}
P(w_n'>t \mid w_1', \dots,w_{n-1}')=P(y_n>t)P(s_n>t)=e^{-\lbd t}e^{\lbd t}G_n(t)=G_n(t),
\end{align*}
so $(w_n'),(t_n')$ has the same distribution as $(w_n),(t_n)$. Let $Y=\{t_n'\}_{n\ge 1}$. By construction, $Y\supset X\cap [0,\zeta)$, which is the claimed stochastic domination.
\end{proof}

\begin{proof}[Proof of Lemma \ref{lem:piAw}]
For tidiness we'll suppress the subscript $\lbd$ from $w_\lbd,\tilde \pi_\lbd$ writing them as $w,\tilde \pi$. We'll prove \eqref{eq:piAw}, then statements 1-3 of the lemma. We begin by representing $A(w)$ in the context of Lemma \ref{lem:iidthin}. Let $(y_i)$ be an i.i.d.~exponential$(\lambda)$ sequence with partial sums $x_i=\sum_{j\le i}y_j$, so that $X=\{x_i\}_i$ is a Ppp with intensity $\lbd$ on $\R_+$. With $k_w:=\min\{i\colon y_i\ge w\}$, the gap event can be written as $A(w) = \{x_{k_w-1}\le 1-w\}$.
Using Lemma \ref{lem:iidthin} with $B=[w,\infty)$, $k_w$ is geometric$(e^{-\lbd w})$ and conditioned on $k_w$, with $(j_i)$ as in Lemma \ref{lem:iidthin} and $z_i=y_{j_i}$, $j_i=i$ for $i<k_w$ and $x_{k_w-1}=\sum_{i<k_w}z_i$, where $(z_i)$ has distribution 
\begin{align}\label{eq:y_cnd}
P(z_i\le y)=\frac{1-e^{-\lambda y}}{1-e^{-\lambda w}}\wedge 1.
\end{align}
Let $u_n=\sum_{i=1}^nz_i$ and $N(x) = \max\{n\colon u_n \le x\}$, then
\begin{align}\label{eq:Awform}
A(w)=\{N(1-w)\ge k_w-1\}.
\end{align}
Again by Lemma \ref{lem:iidthin}, $(z_i)$ is independent of $k_w$, so the process $N$, and in particular $N(1-w)$, are independent of $k_w$.\\

Next we show \eqref{eq:piAw}. We'll first estimate $N(1-w)$, then the conditional probability $P(k_w-1\le N(1-w))$. Let $\mu_\lambda ,\sigma^2_\lambda $ denote the mean and variance of the distribution \eqref{eq:y_cnd}. As $\lambda \to\infty$, if $\lambda w\to\infty$ then $\lambda \mu_\lambda \to 1$ and $\lambda ^2\sigma^2_\lambda  \to 1$, the mean and variance of exponential$(1)$. For $\ep>0$ let $n_w^-(\ep)=\lfloor \lambda (1-w-\ep) \rfloor$ and $n_w^+(\ep) = \lfloor \lambda (1-w+\ep)\rfloor$. Chebyshev's inequality implies that $P(u_{n_w^-(\ep)} \le 1-w),P(u_{n_w^+(\ep)} > 1-w) \to 1$ as $\lambda \to\infty$. In particular, $P(n_w^-(\ep)\le N(1-w)\le n_w^+(\ep)) \to 1$ as $\lambda \to\infty$. Next, let $C_i=\{N(1-w) \in I_i\}$ with $I_0=[0,n_w^-(\ep))$, $I_1=[n_w^-(\ep),n_w^+(\ep)]$, $I_2 = (n_w^+(\ep), m\lambda )$ and $I_3=[m\lambda ,\infty)$ for $m$ to be determined. The above calculation gives $P(C_1)\to 1$ as $\lbd\to\infty$. Decomposing over $(C_i)_{i=0}^3$ and using \eqref{eq:Awform},
\[\tilde \pi(A(w)) = \sum_{i=0}^3 P(N(1-w)\ge k_w-1 \mid C_i)P(C_i).\]
For $n\in \N$,
\begin{align*}
P(N(1-w)\ge k_w-1 \mid N(1-w)=n) &= P(k_w-1 \le n) \\
&= 1-P(k_w >n+1)=1-(1-e^{-\lambda w})^{n+1}.
\end{align*}
As $\lambda \to\infty$, if $(n+1)e^{-\lambda w} \to 0$ then the above $\sim (n+1)e^{-\lambda w}$. In particular, if $\liminf_\lbd \lbd w / \log \lbd>1 $ then letting $\ep\to 0$ slowly as $\lambda \to\infty$,
\[P(N(1-w)\ge k_w-1 \mid C_1) \sim \lambda (1-w)e^{-\lambda w},\]
while for $i\in \{0,2\}$,
\[P(N(1-w)\ge k_w-1 \mid C_i) = O(m\lambda e^{-\lambda w}).\]
Since $P(C_1)\to 1$, $P(C_i)\to 0$ for $i \ne 1$. Using this and the above display for $i\in \{0,2\}$ as well as $P(N(1-w)\ge k_w-1 \mid C_3)\le 1$, if we can find a constant $m$ such that $P(C_3) = o(\lambda e^{-\lambda w})$ then \ref{eq:piAw} holds. To do this we use a large deviations estimate. Note that $C_3=\{u_{\lceil m\lambda  \rceil} \le 1-w\}$. The distribution in \eqref{eq:y_cnd} has density function $\lambda e^{-\lambda y}/(1-e^{-\lambda w})$ for $y\in [0,w]$. If $Y$ has this distribution and $\theta>0$ then integrating,
\begin{align}\label{eq:Ymgf}
E[e^{-\theta Y}] = \frac{\lambda }{\lambda +\theta}\frac{1-e^{-(\lambda +\theta)w}}{1-e^{-\lambda w}}.
\end{align}
Thus for a given $n,t$ and any $\theta>0$,
\[P(u_n \le t) = P( e^{-\theta x_n} \ge e^{-\theta t}) \le e^{\theta t} E[e^{-\theta u_n}] = e^{\theta t}(E[e^{-\theta Y}])^n.\]
If $ne^{-\lambda w} \to 0$ the second fraction in \eqref{eq:Ymgf} has $\limsup \le 1$ and
\[\limsup_\lbd P(u_n \le t) \le e^{\theta t}(1+\theta/\lambda )^{-n}.\]
The right-hand side, minimized with respect to $\theta$, gives $\exp(n-\lambda t - n\log(n/(\lambda t)))$. If $n=m\lambda $ the exponent is $\lambda (m-t)-m\lambda \log (m/t) = m\lambda (1-\log(m/t))- t\lambda $. If $m/t\ge 2e$ the above is at most $-m\lambda $. Taking $t=1-w$ and $m=2e$ we find that $\limsup_\lbd P(C_3) \le e^{-2e\lambda}=o(\lambda e^{-\lambda w})$ as desired, completing the demonstration of \eqref{eq:piAw}.\\

Next we establish points 1 and 2 from the statement of the lemma, recalling $v_\ell=\inf\{x\colon X\cap (x,x+w)=\emptyset\}$ and $v_r=\inf(X\cap(v_\ell,\infty))$, where $X$ is a Ppp with intensity $\lbd$ on $\R_+$. By definition, $v_\ell = x_{k_w-1} = \sum_{i<k_w}z_i$ and $v_r=x_{k_w}$ so $v_r-v_\ell=y_{k_w}$. Using Lemma \ref{lem:iidthin}, conditioned on $k_w$ and $(z_i)$, which together determine $v_\ell$, $y_{k_w}$ has the distribution of exponential$(\lbd)$ conditioned on being $\ge w$, so point 1 follows from the memoryless property.\\

For point 2 it suffices to show that (a) conditioned on $A(w_\lbd)$, $X\cap [0,v_\ell]$ dominates a Ppp on $[0,v_\ell]$ with intensity $\lbd$ and that (b) conditioned on $v_\ell,v_r$ and $X\cap [0,v_\ell]$ (which determine $A(w_\lbd)$), $X\cap (v_r,\infty)$ is a Ppp on $(v_r,\infty)$ with intensity $\lbd$ independent of $X\cap[0,v_\ell]$. We begin with (b). Since $k_w$ is a stopping time for the Markov chain $(y_i)$ and $(y_i)$ are i.i.d., conditioned on $k_w$ and $(y_1,\dots,y_{k_w})$, which determine $v_\ell,v_r$ and $X\cap [0,v_\ell]$, $(y_{k_w+i})_{i\ge 1}$ has the same distribution as $(y_i)_{i\ge 1}$, from which (b) follows. For (a), since $X\cap [0,v_\ell] = \{x_i\}_{i=1}^{k_w-1}= \{u_i\}_{i=1}^{k_w-1}$ (using notation from earlier in this proof), by Lemma \ref{lem:ratedom} it suffices to show that 
\begin{align}\label{eq:zcndrate}
\liminf_{s\to 0^+} - s^{-1}\log P(z_n>y+s \mid z_n>y , z_1,\dots,z_{n-1},N(1-w)\ge k_w-1)\ge \lbd
\end{align}
almost surely for all $n,y$. Recall that $(z_i)$ are i.i.d.~with distribution given by \eqref{eq:y_cnd}, so for $y<w$ we compute
\begin{align}\label{eq:zhaz}
- \frac{d}{dy} \log P( z_i > y) =\frac{\lbd}{1-e^{-\lbd(w-y)}} \ge \lbd.
\end{align}
Let $(\F_n)$ be the natural filtration of $(z_n)$, fix $i$, and let 
\[h(y,n)=P(N(1-w) \ge i \mid z_n>y, \F_{n-1}),\]
and note that $y\mapsto h(y,n)$ is non-increasing. Using Bayes' formula,
\begin{align*}
&P(z_n > y+s \mid z_n> y, N(1-w)\ge i, \F_{n-1}) \\
&= \frac{P(N(1-w)\ge i \mid z_n>y+s, \F_{n-1})P(z_n > y+s \mid z_n>y, \F_{n-1})}{P(  N(1-w)\ge i\mid z_n>y, \F_{n-1})}\\
&= \frac{h(y+s,n)}{h(y,n)}P(z_n > y+s \mid z_n>y, \F_{n-1})\\
&\le P(z_n > y+s \mid z_n>y, \F_{n-1})\\
&= P(z_n > y+s \mid z_n>y),
\end{align*}
using the i.i.d~property on the last step. Replacing $i$ with $k_w-1$ and noting $(z_n)$ and $k_w$ are independent, the same inequality holds, i.e., 
\begin{align}\label{eq:znlb}
&P(z_n > y+s \mid z_n> y, N(1-w)\ge k_w-1, \F_{n-1}) \nonumber \\
&\le P(z_n > y+s \mid z_n>y).
\end{align}
Using \eqref{eq:zhaz},
\begin{align}\label{eq:zcndhaz}
&\lim_{s\to 0^+} - s^{-1} \log P(z_n>y+s \mid  z_n>y)  \nonumber\\
&= \lim_{s\to 0^+} - s^{-1}(\log P(z_n>y+s) - \log P(z_n>y)) \nonumber\\
&= -\frac{d}{dy} \log P(z_n>y) \ge \lbd.
\end{align}
Taking $\liminf_{s\to 0^+}-s^{-1} \log (\cdot)$ of both sides of \eqref{eq:znlb} and using \eqref{eq:zcndhaz} gives \eqref{eq:zcndrate} and completes the proof.
\end{proof}

\section{Estimates of $N$}\label{sec:Nest}

\begin{proof}[Proof of Lemma \ref{lem:countest}]
We prove (a)-(d) in order. To prove (a) we will need to show that with $M:=N-\rho$ and $\sigma_-(a):=\inf\{t\colon |M(t)| < a\}$,
\begin{align}\label{eq:Nexp}
P(\sigma_-(a)>t) \le E[\,|M(0)|\, ]e^{-t}/a.
\end{align}
and that with $\sigma_+(a)=\inf\{t\colon |M(t)|\ge a\}$,
\begin{align}\label{eq:Nexp1}
\liminf_\rho P(\tau(\lf \rho\rf) < \sigma_+(2\sqrt{\rho}) \mid |M(0)|<\sqrt{\rho}) \ge 1/2,
\end{align}
and for some $c,C>0$ and all $\rho,t$,
\begin{align}\label{eq:Nexp2}
P(\sigma_+(2\sqrt{\rho})>t ) \le C_1e^{-c_1t}.
\end{align}
From these estimates, (a) is established as follows. Define the alternating sequence of stopping times $t_0=0$, $t_{2i+1} = \inf\{t>t_{2i}\colon |M(t)|<\sqrt{\rho}\}$ and $t_{2i} = \inf\{t>t_{2i-1}\colon |M(t)| \ge 2\sqrt{\rho}\}$ and note that $|M(t_i)|\le 2\sqrt{\rho}+1$ for $i\ge 1$. From \eqref{eq:Nexp} with $a = \sqrt{\rho}$,
\begin{align*}
P(\sigma_-(\sqrt{\rho})>t ) \le E[\,\rho^{-1/2}|M(0)|\,] e^{-t},
\end{align*}
from which we need two expressions, the first obtained by shifting $t$ by $w_0:=\log (E[ \,\rho^{-1/2} |M_0|\, ])$ and the second by restricting $|M(0)|$:
\enumar
\item $P(\sigma(a)>t + w_0) \le e^{-t}$ and
\item $P(\sigma_-(\sqrt{\rho})>t \mid |M(0)| \le 2\sqrt{\rho}+1) \le (2+o(1))e^{-t}.$
\enumend
With $s_i=t_i-t_{i-1}$, for $\theta<c_1 \wedge 1$,  there is $1<C_2<\infty$ such that, using 1 above, $E[e^{\theta (t_1-w_0)_+}] \le C_2$, and using \eqref{eq:Nexp2} for even $i$ and 2 above for odd $i$, $E[e^{\theta s_i}] \le C_2$ for all $i>1$. In particular,
\[E[e^{\theta(t_{2i}-w_0)_+}] \le (C_2)^{2i}.\]
Markov's inequality then implies
$P(t_{2i}\ge w_0+4\log(C_2) i / \theta) \le (C_2)^{-2i}$, so that given $t$, with $i=\lf \theta t / 4\log(C_2)\rf$, $P(t_{2i}\ge w_0+t) \le C_2^2e^{-\theta t/2} $. Let $I=\inf\{i\colon \tau(\lf\rho\rf)< t_{2i}\}$, so that from \eqref{eq:Nexp1}, $P(I>i)\le (1/2+o(1))^i \le (3/4)^i$ for large $\rho$. Given $t$, taking $i$ as above,
\begin{align*}
P(\tau(\lf\rho\rf)>t+w_0) &\le P(t_{2I} > w_0+t) \\
&\le P(I>i) + P(t_{2i}>w_0+t) \\
&\le (3/4)^i + C_2^2e^{-\theta t/2} \\
&\le (4/3)e^{-\theta t / 4\log(C_2)} + C_2^2e^{-\theta t/2},
\end{align*}
and (a) holds with $c=\theta \min(1/2,1/(4\log C_2))$ and $C=(4/3)+C_2^2$. We now establish \eqref{eq:Nexp},\eqref{eq:Nexp1} and \eqref{eq:Nexp2}. For a process $X$ with compensator $X^p$ (i.e., $X^p$ is predictable and $X-X^p$ is a local martingale), when $X^p$ is almost surely absolutely continuous, $d X^p(t)/dt$ is called the drift of $X$. From its transition rates, the process $N$ has drift $\rho-N$, so $M:=N-\rho$ has drift $-M$. Since $M(t)$ and the function $e^t$ have zero quadratic covariation, by the product rule for semimartingales, see \cite[Lemma 3.4]{foxall_naming_2018}, $\phi(t):=e^tM(t)$ has drift $e^t M(t) + e^t(-M(t)) = 0$, i.e., $\phi$ is a martingale. Since $M$ jumps by at most $1$, the sign of $M$ is constant for $t\in [0,\sigma(1)]$ where $\sigma(a)=\inf\{t\colon |M(t)| < a\}$, and in particular, $|\phi(\cdot \wedge \sigma(a))|$ is a non-negative martingale for $a\ge 1$. If $|M(s)|\ge a$ for all $s\le t$ then $\sigma(a)>t$ and $|\phi(t \wedge \sigma(a))|=|\phi(t)| = e^t |M(t)|\ge e^ta$. Since $|\phi(0)|=|M(0)|$, Doob's inequality then gives
\begin{align*}
P(\sigma(a)>t) &= P(|M(s)|\ge a \ \forall s\le t) \nonumber \\
&\le P(\sup_s|\phi(s \wedge \sigma(a))| \ge e^ta) \le E[\,|M(0)|\, ]e^{-t}/a,
\end{align*}
which is \eqref{eq:Nexp}. 

For \eqref{eq:Nexp1},\eqref{eq:Nexp2} we need the notion of a natural scale, which is a function $y:\N\to \R$ such that $y(N)$ is a martingale. Solving for $Ay=0$ where $A$ is the generator of $N$, $y$ has the form $y(n)=\sum_{i=1}^n z(i)$ where $z(n+1)/z(n) = n/\rho$ is the reciprocal of the ratio of transition rates. Take $z(\lf\rho\rf)=1$. For \eqref{eq:Nexp1}, notice that $z$ is decreasing on $[0,\lf\rho\rf]$ and increasing on $[\lceil\rho\rceil,\infty)$, so
\begin{align}\label{eq:yhrm1}
y(\lceil \rho - \sqrt{\rho}\,\rceil) - y(\lf\rho -  \,2\sqrt{\rho}\,\rf) \ge y(\lf\rho\rf) - y(\lceil \rho - \sqrt{\rho}\,\rceil) \  \text{and} \nonumber \\
y(\lceil\rho + 2\sqrt{\rho}\,\rceil) - y(\lf\rho +  \sqrt{\rho}\rf) \ge y(\lf\rho +  \sqrt{\rho}\rf) - y(\lceil\rho\rceil).
\end{align}
Let $T = \tau(\lf\rho\rf) \wedge \sigma_+(2\sqrt{\rho})$ and define the stopped process $Y(t):=y(N(t\wedge T))$. For each $\rho$, $Y$ is bounded and $E[T]<\infty$ ($T$ is the hitting time of a set of accessible states in a finite state Markov chain) so optional stopping can be used and gives $E[Y(T)]=E[Y(0)]$. If $T=\tau(\lf\rho\rf)$ then $|N(T)-\rho|<1$ while if $T=\sigma_+(2\sqrt{\rho})$ then $|N(T)-\rho|\ge 2\sqrt{\rho}$. \\

If $|M(0)|<\rho$ $|N(0)-\rho|<\sqrt{\rho}$. We consider separately the cases $\lf\rho\rf < N(0) \le \lf \rho + \sqrt{\rho}\rf$ and $\lceil \rho-\sqrt{\rho} \,\rceil \le N(0)\le \lf\rho\rf$. Let $p = P(T=\tau(\lf\rho\rf)$, noting that $1-p=P(T=\sigma_+(2\sqrt{\rho})$. In the first case $E[Y(0)] \le y(\lf \rho + \sqrt{\rho}\rf)$, so on line 3, adding and subtracting $y(\lceil \rho\rceil)p$ and using the second line of \eqref{eq:yhrm1},
\begin{align*}
0 &= E[Y(T)]-E[Y(0)] \\
&\ge (y(\lf\rho\rf) - y(\lf \rho + \sqrt{\rho}\rf ))p + (y(\lceil \rho + 2\sqrt{\rho}\,\rceil) - y(\lf \rho + \sqrt{\rho}\rf ))(1-p) \\
&\ge (y(\lf\rho\rf)-y(\lceil\rho\rceil))p + (y(\lf \rho + \sqrt{\rho}\rf ) - y(\lceil \rho\rceil ))(1-2p).
\end{align*}
With $z(\lf\rho\rf)=1$, $z(\lceil\rho\rceil)=\lf\rho\rf/\rho \le 1$ so $|y(\lceil\rho\rceil)-y(\lf\rho\rf)| \le 1$ and since $z$ is increasing on $[\lceil\rho\rceil,\infty)$, $d:=y(\lf \rho + \sqrt{\rho}\rf ) - y(\lceil \rho\rceil )\ge \sqrt{\rho}-1$, so
\[1-d \ge - 2pd,\]
giving $p\ge (d-1)/2d = 1/2 - 1/(2d) = 1/2-o(1)$ as $\rho\to\infty$. In the second case $E[Y(0)] \ge y(\lceil \rho - \sqrt{\rho}\rceil)$; doing an analogous calculation as before and using the first line of \eqref{eq:yhrm1},
\begin{align*}
0 &= E[Y(T)]-E[Y(0)] \\
&\le (y(\lf\rho\rf) - y(\lceil \rho - \sqrt{\rho}\,\rceil ))p + (y(\lf \rho - 2\sqrt{\rho}\,\rceil) - y(\lceil \rho - \sqrt{\rho}\,\rceil ))(1-p) \\
&\le (y(\lf\rho\rf) - y(\lceil \rho - \sqrt{\rho}\,\rceil ))(2p-1),
\end{align*}
and since $y(\lf\rho\rf) \ge y(\lceil \rho - \sqrt{\rho}\,\rceil)$, $2p-1\ge 0$, i.e., $p\ge 1/2$, which completes the demonstration of \eqref{eq:Nexp1}. For \eqref{eq:Nexp2}, if $|n-\rho| \le 2\sqrt{\rho}$ then $|z(n+1)/z(n)- 1| \le \rho^{-1/2}$, so taking  $z(\lf\rho\rf)=1$, for $-2\sqrt{\rho} \le n-\rho\le 2\sqrt{\rho}+1$,
\[(1-\rho^{-1/2})^{2\sqrt{\rho}} \le |z(n)-1 | \le (1 + \rho^{-1/2})^{2\sqrt{\rho}+2},\]
and in particular, $\limsup_\rho \sup_{|n-\rho|\le 2\sqrt{\rho}} | \log(z(n))|=:C_3 < \infty$. Now, $y(N)$ has predictable quadratic variation
\[\langle y(N(t))\rangle = \int_0^t (\rho z(N(s)+1)^2 + N(s) z(N(s))^2)ds.\]
Defining now $Y(t):=y(N(t\wedge \sigma_+(2\sqrt{\rho})))$, $Y(t)$ has predictable quadratic variation at least $\rho e^{-2C_3} t\wedge \sigma_+(2\sqrt{\rho})$ for large $\rho$, so with $c_2=e^{-2C_3}$, $(Y(t)-Y(0))^2 - c_2 \rho t\wedge \sigma_+(2\sqrt{\rho})$ is a submartingale. For the same reasons as before, optional stopping can be used for $Y$ at time $\sigma_+(2\sqrt{\rho})$, and gives
\[E[(Y(\sigma_+(2\sqrt{\rho})) - Y(0))^2] - c_2 \rho E[\sigma_+(2\sqrt{\rho})] \ge 0.\]
Using the bound on $|\log(z_n)|$, for $|n-\rho| \wedge |m-\rho|\le 2\sqrt{\rho}+1$ and large $\rho$, $|y(m)-y(n)| \le e^{2C_3} |m-n|$ and
\[E[\sigma_+(2\sqrt{\rho})] \le (c_2 \rho ) ^{-1} e^{2C_3} (2(2\sqrt{\rho}+1))^2 = O(1).\]
Let $C_4=\limsup_\rho E[\sigma_+(2\sqrt{\rho})] < \infty$. Using Markov's inequality,
\[P(\sigma_+(2\sqrt{\rho}) \ge 2C_4) \le 1/2,\]
then using the Markov property and iterating, for integer $k>0$,
\[P(\sigma_+(2\sqrt{\rho}) \ge 2C_4k) \le (1/2)^k,\]
from which \eqref{eq:Nexp2} easily follows, completing the proof of (a).\\

For (b)-(d) we again use the natural scale $y$. Let $q_n=P(\tau(0)<\tau(\lf\rho\rf)\mid N(0)=n )$. Optional stopping applied to $y(N)$ implies
\[y(n) = q_n y(0) + (1-q_n) y(\lf\rho\rf).\]
This time, take $y(0)=0$. Then
\[q_n = \frac{y(n)-y(\lf\rho\rf)}{y(0)-y(\lf\rho\rf)} = \frac{y(\lf\rho\rf)-y(n)}{y(\lf\rho\rf)} = \frac{\sum_{k=n+1}^{\lf\rho\rf} z(k)}{\sum_{k=1}^{\lf\rho\rf} z(k)}.\]
Taking $z(1)=1$ and recalling $z(k+1)/z(k)=k/\rho$ we compute 
\[\log z(k) = -(k-1)\log \rho + \sum_{i=1}^{k-1} \log i.\]
Bounding the sum by an integral and noting $\log 3 > 1$,
\begin{align*}
\sum_{i=1}^{k-1} \log i &= \log 2 + \sum_{i=3}^{k-1} \log i \\
&\le \log 2 + k(\log k -1)- 3(\log 3 - 1) \\
&\le \log 2 + k(\log k - 1).
\end{align*}
Plugging in and highlighting the terms we'll need, for $2\le k\le \lf\rho\rf$,
\begin{align*}
\log z(k) &\le -\log \rho - (k-2)\log \rho/k -k + \log 2 + 2\log k \\
&\le -\log \rho -k + \log (2k^2).
\end{align*}
For (b) we need to show that $\inf_{n\ge 1} q_n = O(1/\rho)$. Since $z(1)=1$ by assumption, $\sum_{k=1}^{\lf\rho\rf} z(k) \ge 1$ so for $n\ge 1$,
\[q_n \le \sum_{k=n+1}^{\lf\rho\rf} z(k) \le \rho^{-1}\sum_{k=2}^\infty 2k^2 e^{-k}= O(1/\rho).\]
For (c) we want $\inf_{n\ge \alpha \rho} q_n = O(e^{-c\rho})$ for some $c>0$ that may depend on $\alpha$ but not $\rho$; for large $\rho$ and $n\ge \alpha \rho$,
\[q_n \le \sum_{k=n+1}^{\lf\rho\rf} z(k) \le \sum_{k=\lf\alpha \rho\rf}^{\lf\rho\rf} 2k^2 e^{-k}= O(e^{-\alpha \rho/2}).\]
For (d), note that if $N(0)=\lf\rho\rf$ the return times $(t_n)_{n\ge 1}$ defined by $t_0=0$ and 
\[t_{n+1}=\inf\{t>t_n\colon N(t)=\lf\rho\rf, \ N(s)\ne \lf\rho\rf \ \text{for some} \ s\in (t_n,t)\}\]
have $t_n-t_{n-1}\ge w_n:=\inf\{t>t_{n-1} \colon N(t)\ne \lf\rho\rf\}$, the time to jump out of state $\lf\rho\rf$ after the $n-1^{\text{st}}$ return. Letting $M=\inf\{n\colon N(t)=0 \ \text{for some} \ t\in (t_{n-1},t_n)\}$, the first excursion that hits $0$,
$\tau(0) \ge t_{M-1} \ge \sum_{n=1}^{M-1} w_n$. Now, $M$ is geometrically distributed with success probability $P(\tau(0)<t_1 \mid N(0)=\lf\rho\rf) \le P(\tau(0)<\tau(\lf\rho\rf) \mid N(0) \ge \rho/2) \le Ce^{-c\rho}$ for some $c,C>0$, by (c), so a union bound gives $P(M-1\le \lf e^{c\rho/2} \rf) \le Ce^{-c\rho}(e^{c\rho/2}+1) = O(e^{-c\rho/2})$. Let $a_n=\1(w_n \ge 1/\rho)$, so that $(a_n)$ are i.i.d.~Bernoulli$(p_\rho)$ with $\liminf_\rho p_\rho := p>0$, since $w_n$ is exponential$(\rho+\lf\rho\rf)$. A large deviations bound for binomial gives $P(\sum_{n=1}^{\lf e^{c\rho/2}\rf} a_n < (p/2) \lf e^{c\rho/2} \rf) \le C_1e^{-c_1 \lf e^{c\rho/2} \rf }= O(e^{-c_1\rho})$ for some $c_1,C_1>0$. The union of these two events has probability at most $C_2e^{-c_2\rho}$ for some $c_2,C_2>0$ and on its complement, $t_{M-1} \ge (1/\rho) \lf e^{c\rho/2}\rf \ge e^{c\rho/3}$ for large $\rho$, completing the proof of (d).
\end{proof}

\paragraph{Acknowledgments:} This research is supported by NSERC Discovery Grants RGPIN-2020-05348 and RGPIN-2024-04653. We thank Khanh Dao-Duc for motivating this project and for his input during the redaction of the manuscript. We thank Jay Newby for initial discussions, guesses and estimations of the hitting time.

\bibliographystyle{plain}
\bibliography{bibliography_arxiv}

@article{garcia2006spatial,
  title={Spatial birth and death processes as solutions of stochastic equations},
  author={Garcia, Nancy L and Kurtz, Thomas G},
  journal={Alea},
  volume={1},
  pages={281--303},
  year={2006}
}

@article{sigman1990one,
  title={One-dependent regenerative processes and queues in continuous time},
  author={Sigman, Karl},
  journal={Mathematics of Operations Research},
  volume={15},
  number={1},
  pages={175--189},
  year={1990},
  publisher={INFORMS}
}

@article{bezborodov2017maximal,
  title={Maximal irreducibility measure for spatial birth-and-death processes},
  author={Bezborodov, Viktor and Di Persio, Luca},
  journal={Statistics \& Probability Letters},
  volume={125},
  pages={25--32},
  year={2017},
  publisher={Elsevier}
}

@article{bezborodov2022spatial,
  title={Spatial birth-and-death processes with a finite number of particles},
  author={Bezborodov, Viktor and Di Persio, Luca},
  journal={Modern Stochastics: Theory and Applications},
  volume={9},
  number={3},
  pages={279--312},
  year={2022},
  publisher={VTeX: Solutions for Science Publishing}
}

@article{mountford2016exponential,
  title={Exponential extinction time of the contact process on finite graphs},
  author={Mountford, Thomas and Mourrat, Jean-Christophe and Valesin, Daniel and Yao, Qiang},
  journal={Stochastic Processes and their Applications},
  volume={126},
  number={7},
  pages={1974--2013},
  year={2016},
  publisher={Elsevier}
}

@article{chan1998large,
  title={Large deviations and quasi-stationarity for density-dependent birth-death processes},
  author={Chan, Terence},
  journal={The ANZIAM Journal},
  volume={40},
  number={2},
  pages={238--256},
  year={1998},
  publisher={Cambridge University Press}
}

@article{lazarescu2019large,
  title={Large deviations and dynamical phase transitions in stochastic chemical networks},
  author={Lazarescu, Alexandre and Cossetto, Tommaso and Falasco, Gianmaria and Esposito, Massimiliano},
  journal={The Journal of Chemical Physics},
  volume={151},
  number={6},
  year={2019},
  publisher={AIP Publishing}
}

@article{assaf2017wkb,
  title={{WKB} theory of large deviations in stochastic populations},
  author={Assaf, Michael and Meerson, Baruch},
  journal={Journal of Physics A: Mathematical and Theoretical},
  volume={50},
  number={26},
  pages={263001},
  year={2017},
  publisher={IOP Publishing}
}

@article{preston1975spatial,
  title={Spatial birth and death processes},
  author={Preston, Chris},
  journal={Advances in applied probability},
  volume={7},
  number={3},
  pages={465--466},
  year={1975},
  publisher={Cambridge University Press}
}

@article{asselah1997sharp,
  title={Sharp estimates for the occurrence time of rare events for symmetric simple exclusion},
  author={Asselah, Amine and Dai Pra, Paolo},
  journal={Stochastic processes and their applications},
  volume={71},
  number={2},
  pages={259--273},
  year={1997},
  publisher={Elsevier}
}

@article{asselah1997occurrence,
  title={Occurrence of rare events in ergodic interacting spin systems},
  author={Asselah, Amine and Dai Pra, Paolo},
  journal={Annales de l'Institut Henri Poincare (B) Probability and Statistics},
  volume={33},
  number={6},
  pages={727--751},
  year={1997},
  publisher={Elsevier}
}

@article{iscoe1994asymptotics,
  title={Asymptotics of exit times for Markov jump processes I},
  author={Iscoe, Ian and McDonald, David R},
  journal={The Annals of Probability},
  pages={372--397},
  year={1994},
  publisher={Institute of Mathematical Statistics}
}

@article{fernandez2016conditioned,
  title={Conditioned, quasi-stationary, restricted measures and escape from metastable states},
  author={Fernandez, Roberto and Manzo, Francesco and Nardi, Francesca Romana and Scoppola, Elisabetta and Sohier, Julien},
  journal={Annals of Applied Probability},
  volume={26},
  number={2},
  pages={760--793},
  year={2016},
  publisher={Institute of Mathematical Statistics}
}

@article{gutmart,
title = {The weak law of large numbers for arrays},
journal = {Statistics \& Probability Letters},
volume = {14},
number = {1},
pages = {49-52},
year = {1992},
issn = {0167-7152},
doi = {https://doi.org/10.1016/0167-7152(92)90209-N},
url = {https://www.sciencedirect.com/science/article/pii/016771529290209N},
author = {Allan Gut}
}

@article{foxall_naming_2018,
	title = {The naming game on the complete graph},
	author = {Foxall, Eric},
	journal = {Electronic Journal of Probability},
	volume = {23},
	issn = {1083-6489},
	url = {https://projecteuclid.org/journals/electronic-journal-of-probability/volume-23/issue-none/The-naming-game-on-the-complete-graph/10.1214/18-EJP250.full},
	doi = {10.1214/18-EJP250},
	shortjournal = {Electron. J. Probab.},
	urldate = {2025-09-23},
	year = {2018},
	langid = {english},
}

@book{lieshout_theory_2019,
	location = {Boca Raton London New York},
	title = {Theory of spatial statistics: a concise introduction},
	isbn = {978-0-367-14642-9 978-0-429-62703-3 978-0-429-62539-8 978-0-429-05286-6 978-0-429-62867-2},
	shorttitle = {Theory of spatial statistics},
	pagetotal = {1},
	publisher = {{CRC} Press},
	author = {Lieshout, Marie-Colette van},
	year = {2019},
	langid = {english},
}

@book{tijms2003first,
  title={A first course in stochastic models},
  author={Tijms, Henk C},
  year={2003},
  publisher={John Wiley \& Sons}
}

@book{illian2008statistical,
  title={Statistical analysis and modelling of spatial point patterns},
  author={Illian, Janine and Penttinen, Antti and Stoyan, Helga and Stoyan, Dietrich},
  year={2008},
  publisher={John Wiley \& Sons}
}

@article{benevs1957fluctuations,
  title={Fluctuations of telephone traffic},
  author={Bene{\v{s}}, V{\'a}clav E},
  journal={Bell System Technical Journal},
  volume={36},
  number={4},
  pages={965--973},
  year={1957},
  publisher={Wiley Online Library}
}

@book{chiu2013stochastic,
  title={Stochastic geometry and its applications},
  author={Chiu, Sung Nok and Stoyan, Dietrich and Kendall, Wilfrid S and Mecke, Joseph},
  year={2013},
  publisher={John Wiley \& Sons}
}

@article{omwonylee2020general,
  title={General large deviations of longest gaps in homogeneous Poisson processes},
  author={Omwonylee, Joseph Okello and Yang, Xiangfeng},
  journal={Journal of Mathematical Analysis and Applications},
  volume={489},
  number={2},
  pages={124194},
  year={2020},
  publisher={Elsevier}
}

@inproceedings{fan2000asymptotic,
  title={Asymptotic properties of the maximal subinterval of a Poisson process},
  author={Fan, Ruzong and Lange, Kenneth },
  booktitle={Stochastic Processes, Physics and Geometry: New Interplays. II: A Volume in Honor of Sergio Albeverio, Canadian Mathematical Society Conference Proceedings},
  volume={29},
  pages={175--187},
  year={2000}
}

@article{asmussen2008asymptotic,
  title={Asymptotic behavior of total times for jobs that must start over if a failure occurs},
  author={Asmussen, S{\o}ren and Fiorini, Pierre and Lipsky, Lester and Rolski, Tomasz and Sheahan, Robert},
  journal={Mathematics of Operations Research},
  volume={33},
  number={4},
  pages={932--944},
  year={2008},
  publisher={INFORMS}
}

@book{cressie2015statistics,
  title={Statistics for spatial data},
  author={Cressie, Noel},
  year={2015},
  publisher={John Wiley \& Sons}
}

@book{daley2008introduction,
  title={An introduction to the theory of point processes: volume II: general theory and structure},
  author={Daley, Daryl J and Vere-Jones, David},
  year={2008},
  publisher={Springer}
}

@article{glynn2011exponential,
  title={On exponential limit laws for hitting times of rare sets for Harris chains and processes},
  author={Glynn, Peter W},
  journal={Journal of Applied Probability},
  volume={48},
  number={A},
  pages={319--326},
  year={2011},
  publisher={Cambridge University Press}
}

@article{fonseca2017dynamic,
  title={Dynamic multiscale spatiotemporal models for Poisson data},
  author={Fonseca, Tha{\'\i}s CO and Ferreira, Marco AR},
  journal={Journal of the American Statistical Association},
  volume={112},
  number={517},
  pages={215--234},
  year={2017},
  publisher={Taylor \& Francis}
}

@article{zweimuller2019hitting,
  title={Hitting-time limits for some exceptional rare events of ergodic maps},
  author={Zweim{\"u}ller, Roland},
  journal={Stochastic Processes and their Applications},
  volume={129},
  number={5},
  pages={1556--1567},
  year={2019},
  publisher={Elsevier}
}

@article{fernandez2015asymptotically,
  title={Asymptotically exponential hitting times and metastability: a pathwise approach without reversibility},
  author={Fernandez, Roberto and Manzo, Francesco and Nardi, Francesca and Scoppola, Elisabetta},
  year={2015},
  journal={Electronic Journal of Probability},
  volume={20},
}

@article{ogata1998space,
  title={Space-time point-process models for earthquake occurrences},
  author={Ogata, Yosihiko},
  journal={Annals of the Institute of Statistical Mathematics},
  volume={50},
  number={2},
  pages={379--402},
  year={1998},
  publisher={Springer}
}

@inproceedings{gu2021spatio,
  title={The spatio-temporal Poisson point process: A simple model for the alignment of event camera data},
  author={Gu, Cheng and Learned-Miller, Erik and Sheldon, Daniel and Gallego, Guillermo and Bideau, Pia},
  booktitle={Proceedings of the IEEE/CVF International Conference on Computer Vision},
  pages={13495--13504},
  year={2021}
}

@article{sankararaman2017spatial,
  title={Spatial birth--death wireless networks},
  author={Sankararaman, Abishek and Baccelli, Fran{\c{c}}ois},
  journal={IEEE Transactions on Information Theory},
  volume={63},
  number={6},
  pages={3964--3982},
  year={2017},
  publisher={IEEE}
}

@incollection{keilson1979rarity,
  title={Rarity and exponentiality},
  author={Keilson, Julian},
  booktitle={Markov Chain Models—Rarity and Exponentiality},
  pages={130--163},
  year={1979},
  publisher={Springer}
}

@article{glasserman1995limits,
  title={Limits of first passage times to rare sets in regenerative processes},
  author={Glasserman, Paul and Kou, Shing-Gang},
  journal={The Annals of Applied Probability},
  pages={424--445},
  year={1995},
  publisher={Institute of Mathematical Statistics}
}

@article{benois2013hitting,
  title={Hitting times of rare events in Markov chains},
  author={Benois, Olivier and Landim, Claudio  and Mourragui, Mustapha},
  journal={Journal of Statistical Physics},
  volume={153},
  number={6},
  pages={967--990},
  year={2013},
  publisher={Springer}
}

@article{cogburn1985distribution,
  title={On the distribution of first passage and return times for small sets},
  author={Cogburn, Robert},
  journal={The Annals of Probability},
  pages={1219--1223},
  year={1985},
  publisher={Institute of Mathematical Statistics}
}

@article{deheuvels1985erdos,
  title={On the Erd{\"o}s-R{\'e}nyi theorem for random fields and sequences and its relationships with the theory of runs and spacings},
  author={Deheuvels, Paul},
  journal={Probability Theory and Related Fields},
  volume={70},
  number={1},
  pages={91--115},
  year={1985},
  publisher={Springer}
}

@article{zweimuller2022hitting,
  title={Hitting times and positions in rare events},
  author={Zweim{\"u}ller, Roland},
  journal={Annales Henri Lebesgue},
  volume={5},
  pages={1361--1415},
  year={2022}
}

@article{abadi2004sharp,
  title={Sharp error terms and neccessary conditions for exponential hitting times in mixing processes},
  author={Abadi, Miguel},
  journal={The Annals of Probability},
  volume={32},
  number={1A},
  pages={243--264},
  year={2004},
  publisher={Institute of Mathematical Statistics}
}

@article{pene2020spatio,
  title={Spatio-temporal Poisson processes for visits to small sets},
  author={P{\`e}ne, Fran{\c{c}}oise and Saussol, Beno{\^\i}t},
  journal={Israel Journal of Mathematics},
  volume={240},
  number={2},
  pages={625--665},
  year={2020},
  publisher={Springer}
}

@article{onaran2023functional,
  title={Functional central limit theorems for local statistics of spatial birth--death processes in the thermodynamic regime},
  author={Onaran, Efe and Bobrowski, Omer and Adler, Robert J},
  journal={The Annals of Applied Probability},
  volume={33},
  number={5},
  pages={3958--3986},
  year={2023},
  publisher={Institute of Mathematical Statistics}
}

@article{schuss2007narrow,
  title={The narrow escape problem for diffusion in cellular microdomains},
  author={Schuss, Zeev and Singer, Amit and Holcman, David},
  journal={Proceedings of the National Academy of Sciences},
  volume={104},
  number={41},
  pages={16098--16103},
  year={2007},
  publisher={National Academy of Sciences}
}

@article{coombs2009diffusion,
  title={Diffusion on a sphere with localized traps: Mean first passage time, eigenvalue asymptotics, and Fekete points},
  author={Coombs, Daniel and Straube, Ronny and Ward, Michael},
  journal={SIAM Journal on Applied Mathematics},
  volume={70},
  number={1},
  pages={302--332},
  year={2009},
  publisher={SIAM}
}

@article{newby2016first,
  title={First-passage time to clear the way for receptor-ligand binding in a crowded environment},
  author={Newby, Jay and Allard, Jun},
  journal={Physical review letters},
  volume={116},
  number={12},
  pages={128101},
  year={2016},
  publisher={APS}
}

@article{Daoduc2010thres,
  title = {Threshold activation for stochastic chemical reactions in microdomains},
  author = {Dao Duc, K. and Holcman, D.},
  journal = {Phys. Rev. E},
  volume = {81},
  issue = {4},
  pages = {041107},
  numpages = {11},
  year = {2010},
  month = {Apr},
  publisher = {American Physical Society},
  doi = {10.1103/PhysRevE.81.041107},
  url = {https://link.aps.org/doi/10.1103/PhysRevE.81.041107}
}
\end{document}